\documentclass{daj}

\usepackage{amssymb, amsthm, amsmath, amsfonts,mathtools,mathrsfs,calc}
\usepackage{array, epsfig}
\usepackage{bbm}
\usepackage{enumitem}
\usepackage{color}
\usepackage{hyperref}
\usepackage[numbers,square]{natbib}

\usepackage{ stmaryrd }
\usepackage{upgreek}
\numberwithin{equation}{section}

\newcommand{\ind}{\mathbbm{1}}
\newcommand{\eps}{\varepsilon}

\def\NN{\mathbb{N}}

\def\bP{\mathbb{P}}

\def\RR{\mathbb{R}}
\def\RRd1{\mathbb{R}^{d+1}}
\def\SS{\mathbb{S}}

\def\ZZ{\mathbb{Z}}

\def\bE{\mathbf{E}}

\def\bP{\mathbf{P}}

\def\fI{\mathfrak{I}}

\def\cA{\mathcal{A}}

\def\cF{\mathcal{F}}

\def\cI{\mathcal{I}}
\def\cJ{\mathcal{J}}

\newcommand{\todistr}{\overset{d}{\underset{d\to\infty}\longrightarrow}}

\newcommand{\toas}{\overset{a.s.}{\underset{d\to\infty}\longrightarrow}}

\newcommand{\ton}{\overset{}{\underset{d\to\infty}\longrightarrow}}

\newcommand{\tosim}{\underset{d\to\infty}{\sim}}

\def\dint{\textup{d}}

\def\Bin{\textup{Bin}}

\def\relint{\textup{relint}}
\DeclareMathOperator{\pos}{pos}

\theoremstyle{plain}
\newtheorem{theorem}{Theorem}[section]
\newtheorem{lemma}[theorem]{Lemma}
\newtheorem{corollary}[theorem]{Corollary}
\newtheorem{proposition}[theorem]{Proposition}

\theoremstyle{definition}

\theoremstyle{remark}
\newtheorem{remark}[theorem]{Remark}

\dajAUTHORdetails{%
  title = {Random Cones in High Dimensions I: Donoho-Tanner and Cover-Efron Cones}, %% please capitalize all significant words
  author = {Thomas Godland, Zakhar Kabluchko and Christoph Th{\"a}le},
  plaintextauthor = {Thomas Godland, Zakhar Kabluchko, Christoph Thäle},
  keywords = {Conic intrinsic volume, conic quermassintegral, Cover-Efron cone, Donoho-Tanner cone, high dimensions, limit theorem, phase transition, random cone, random Gale diagram, statistical dimension, stochastic geometry, threshold phenomenon},
}   %%% END \dajAUTHORdetails

\dajEDITORdetails{%
   year={2022},
   number={5},
   received={14 December 2020},   % received date: example: 7 January 2017
   published={24 June 2022},  % published date
   doi={10.19086/da.36223},       % XXX = number of paper, e.g. da006 for paper#6
}   %%% END \dajEDITORdetails

\begin{document}

\begin{frontmatter}[classification=text]
%% EDITOR: this will force the keywords to appear right after the Abstract.
%%   If the abstract is too long and would force the keywords off the
%%   front page, please comment out % [classification=text] above
%%   This way the keywords will be floated on the bottom of the first page
%%   even though the Abstract spills over to the next page.

%%% AUTHOR: Title goes here.  This line is optional.  You must use it
%%   if title has footnote attached or requires nontrivial typesetting,
%%   e.g., inclusion of linebreaks to force nice layout.
%\title{Short Proof of R\"odl's $n^{\log\log n}$ Bound\footnote{This is a footnote to the title}} %% please capitalize all significant words

%%% AUTHOR:
%%% List all authors. If you wish, place grant acknowledgements in \thanks.
%%% In brackets include a short tag for each author.
\author[TG]{Thomas Godland$^*$}
\author[ZK]{Zakhar Kabluchko\thanks{Supported by the DFG priority program SPP 2265 \textit{Random Geometric Systems} and the DFG under Germany's Excellence strategy EXC 2044 -- 390685587, \textit{Mathematics M\"unster: Dynamics -- Geometry -- Structure}.}}
\author[CT]{Christoph Th\"ale\thanks{Supported by the DFG priority program SPP 2265 \textit{Random Geometric Systems}.}}
%\author[andy]{Andrew Chi-Chih Yao\thanks{Supported by...}}

%%% AUTHOR: Abstract goes here
\begin{abstract}
Two models of random cones in high dimensions are considered, together with their duals. The Donoho-Tanner random cone $D_{n,d}$ can be defined as the positive hull of $n$ independent $d$-dimensional Gaussian random vectors. The Cover-Efron random cone $C_{n,d}$ is essentially defined as the same positive hull, conditioned on the event that it is not the whole space. We consider expectations of various combinatorial and geometric functionals of these random cones and prove that they satisfy limit theorems, as $d$ and $n$ tend to infinity in a suitably coordinated way. This includes, for example, large deviation principles and central as well as non-central limit theorems for the expected number of $k$-faces and the $k$-th conic intrinsic volumes, as $n$, $d$ and possibly also $k$ tend to infinity simultaneously. Furthermore, we determine the precise high-dimensional asymptotic behaviour of the expected statistical dimension for both models of random cones, uncovering thereby another high-dimensional phase transition. As an application, limit theorems for the number of $k$-faces of high-dimensional polytopes generated by random Gale diagrams are discussed as well.
\end{abstract}
\end{frontmatter}

%%% AUTHOR: body of paper starts here
\section{Introduction}

By a polyhedral cone in $\RR^d$, $d\in\NN$, one understands an intersection of finitely many closed half-spaces whose bounding hyperplanes all pass through the origin. Thus, a polyhedral cone (in this paper just called a cone for simplicity, since we only deal with cones of this type) in $\RR^d$ is a set of solutions to a finite system of homogeneous linear inequalities in $d$ variables. In the present paper we are interested in randomly generated cones, which correspond to the set of solutions of a finite random system of homogeneous linear inequalities. In the literature two natural models for random cones have been considered and are in the focus of part I of the present paper as well. To define them, let $\phi$ be an even probability measure on $\RR^d$ which puts zero mass on each $(d-1)$-dimensional linear hyperplane. For example, $\phi$ could be the standard Gaussian distribution on $\RR^d$ or the uniform distribution on the $(d-1)$-dimensional unit sphere. We let $X_1,X_2,\ldots$ be a sequence of independent random points in $\RR^d$ with distribution $\phi$. Then,
\begin{itemize}
\item[(i)] the \textit{Cover-Efron random cone} $C_{n,d}$ in $\RR^d$, introduced by Cover and Efrom~\cite{CoverEfron} and further studied by Hug and Schneider~\cite{HugSchneiderConicalTessellations}, is defined as the positive hull
$$
C_{n,d}:=\pos(X_1,\ldots,X_n) = \Big\{\sum_{i=1}^n \lambda_iX_i: \lambda_1,\dots,\lambda_n\ge 0\Big\}
$$
conditionally on the event that $\pos(X_1,\ldots,X_n)\neq\RR^d$.
\item[(ii)] the \textit{Donoho-Tanner random cone} $D_{n,d}$ in $\RR^d$, introduced by Donoho and Tanner~\cite{DonohoTanner}, is defined as
$$
D_{n,d}:=\pos(X_1,\ldots,X_n)
$$
without any conditioning. Alternatively, we can write
$$
D_{n,d} = A\RR_+^n,\qquad n\in\NN,
$$
where $A$ is the $d\times n$ matrix with columns $X_1,\ldots,X_n$ and $\RR_+^n=[0,\infty)^n$ is the positive orthant in $\RR^n$.
\end{itemize}
We remark  that there is a close relation between the both models, because the Cover-Efron cone is nothing but a Donoho-Tanner cone conditioned on the event that it is not the whole $\RR^d$.

For both models it has been shown that in certain regimes for the dimension $d$ and the number of generating points $n$, the cones are with high probability $k$-neighbourly for a sufficiently slow growth of $k=k(d)$ relative to $d$, meaning that with probability tending to $1$, as $d\to\infty$, every $k$-tuple of the $n$ defining vectors spans a $k$-dimensional face of the cone. On the other hand, if $k$ grows sufficiently fast with $d$, the probability of being $k$-neighbourly goes to $0$ as $d\to\infty$. The motivation for the investigation of such threshold phenomena comes from the Grassmannian approach to linear programming suggested by Vershik in the 1970's and developed by Vershik and Sporyshev in~\cite{vershik_sporyshev_estimation_simplex1983,vershik_sporyshev_asymptotic_estimate_1986,vershik_sporyshev_asymptotic_faces_random_polyhedra1992}. It has subsequently been applied to a number of problems in convex optimization by Donoho and Tanner \cite{donoho_tanner_observed_universality,donoho_tanner_neighborliness,donoho_tanner_sparse_nonnegative_sol,DonohoTanner,donoho_tanner_exponential_bounds} as well as by Amelunxen, Lotz, McCoy and Tropp~\cite{ALMT14}. In this approach, one endows the space of instances of an optimization problem with a natural probability measure and then asks for the probability that the optimization problem is solvable. Threshold phenomena therefore describe transitions from solvability to non-solvability.

For both types of random cones, $C_{n,d}$ and $D_{n,d}$, explicit formulas for the expectations of various combinatorial and geometric functionals exist in the literature. Examples of such functionals include the face numbers, the conic intrinsic volumes as well as the conic quermassintegrals. Using these formulas, Donoho and Tanner~\cite{DonohoTanner} for $D_{n,d}$ and Hug and Schneider~\cite{HugSchneiderThresholdPhenomena} for $C_{n,d}$ proved limit theorems for the expected functionals when the parameters of the problem, that is, $d$ and $n$, diverge to infinity in a suitably coordinated way. In particular, they established a number of threshold phenomena for these high-dimensional random cones which are similar to the ones that we described above. In this connection let us also mention that there are numerous works on threshold phenomena for volumes of random polytopes~\cite{bonnet_etal,bonnet_etal2,chakraborti_etal,dyer_fueredi_mcdiarmid90,dyer_fueredi_mcdiarmid92,pivovarov}. For example, assuming that $n=n(d)$ is such that $d/n\to\delta\in[0,1]$, Hug and Schneider in \cite[Theorem 1.1]{HugSchneiderThresholdPhenomena} obtained that
$$
\lim_{d\to\infty}{\bE f_k(C_{n,d})\over{n\choose k}} = \begin{cases}
1 &: \delta\in(1/2,1],\\
(2\delta)^k &: \delta\in[0,1/2)
\end{cases}
$$
for any fixed $k\in\NN$, where $f_k(C_{n,d})$ denotes the number of $k$-(dimensional) faces of $C_{n,d}$. In this threshold phenomenon the critical case $\delta=1/2$ has been left open. As one of our first results we prove that in this case, the above limit is still equal to $1$, see Theorem \ref{theo:4.2} below. We also look much more carefully into the critical window around $\delta=1/2$ by considering the situation where
$$
n=2d+c\sqrt{d}+o(\sqrt{d}),
$$
as $d\to\infty$, for some parameter $c\in\RR$ and some sequence $o(\sqrt{d})$ which is asymptotically of lower order compared to $\sqrt{d}$. We show that, for any fixed $k\in\NN$,
$$
\lim_{d\to\infty}\sqrt{d}\bigg(1-{\bE f_k(C_{n,d})\over{n\choose k}}\bigg)={e^{-c^2/4}\over\Phi(-c/\sqrt{2})}\cdot{k\over 2\sqrt{\pi}},
$$
where $\Phi(\,\cdot\,)$ stands for the distribution function of a standard Gaussian random variable, see Theorem \ref{theorem:f_k_CoverEfron_k_fixed} below. In the same regime for $n$ we establish in Proposition \ref{prop:DT61220} that for the Donoho-Tanner random cones $D_{n,d}$ one has that
$$
\lim_{d\to\infty}{\bE f_k(D_{n,d})\over{n\choose k}} = \Phi\Big(-{c\over\sqrt{2}}\Big).
$$
We prove similar results also for increasing face dimensions $k=k(d)$, especially in the case where $k$ grows with $d$ in a linearly coordinated way. In addition and as anticipated above, not only the expected number of $k$-faces of $C_{n,d}$ and $D_{n,d}$ is treated, but also the expected $k$-th conic intrinsic volume and the expected $k$-th conic quermassintegral for which we also uncover a number of limit theorems like central and non-central limit theorems or large deviation principles. Generally speaking, our investigations are always driven by the search for regimes of $d$, $n=n(d)$ and $k=k(d)$ for which these combinatorial and geometric parameters of the random cones under investigation exhibit a non-trivial  high-dimensional limiting behaviour, which describes the phase transition of the parameter under consideration. We apply our results also to study the phase transition for the number of $k$-dimensional faces of polytopes in $\RR^d$ generated by a random Gale diagram, as suggested by Schneider \cite{SchneiderGale}. In the background of our results are probabilistic interpretations of the combinatorial and geometric parameters of the Cover-Efron and Donoho-Tanner random cones in terms of probabilities for binomial random variables, which are already present in \cite{DonohoTanner,HugSchneiderThresholdPhenomena}. We can thus build on known limit theorems from probability theory, such as the law of large numbers, the central limit theorem, the local limit theorem, Cram\'er's theorem for large deviations, mod-$\phi$ convergence and the Cram\'er-Petrov theorem for moderate deviations. We believe that this approach is more powerful than the one in \cite{HugSchneiderThresholdPhenomena}, which is based on various intricate inequalities for binomial coefficients.

\medspace

This paper is organized as follows. In Section~\ref{section:preliminaries} we introduce some notation, collect various limit theorems for binomial random variables, and formally define the conic intrinsic volumes and conic quermassintegrals for polyhedral convex cones. Section~\ref{section:random_cones}  formally introduces the above mentioned models of random cones and rephrases some known results which are used in this text. The following Sections~\ref{section:limit_faces}--\ref{section:limit_stat_dimension} contain limit theorems for the expected number of faces, the expected conic intrinsic volumes and the conic quermassintegrals as well as the expected statistical dimension. Section~\ref{section:limit_faces} also contains the application to polytopes generated by random Gale diagrams.

\section{Preliminaries}\label{section:preliminaries}

\subsection{Notation}

In this paper we write $\NN=\{1,2,\ldots\}$ for the natural numbers and $\NN_0=\NN\cup\{0\}=\{0,1,2,\ldots\}$ for the natural numbers including zero. Moreover, we work in a $d$-dimensional Euclidean space $\RR^d$, which is supplied with the standard scalar product $\langle\,\cdot\,,\,\cdot\,\rangle$.

We will write $\Bin(m,p)$ for a binomial random variable with $m\in\NN$ being the number of trials and $p\in[0,1]$ being the success probability. Similarly, $N(0,1)$ denotes a standard normal (Gaussian) random variable. Moreover, we let $\Phi$ denote the probability distribution function of $N(0,1)$, that is,
\begin{align*}
\Phi(x)=\frac{1}{\sqrt{2\pi}}\int_{-\infty}^xe^{-t^2/2} \,\dint t,\qquad x\in\RR.
\end{align*}
We implicitly assume that such random variables are defined over some probability space $(\Omega,\cA,\bP)$, and we denote expectation (integration) with respect to the probability measure $\bP$ by $\bE$. Convergence in distribution and almost sure convergence of a sequence of random variables will be indicated by $\overset{d}{\longrightarrow}$ and $\overset{a.s.}{\longrightarrow}$, respectively. In the situations we consider, the random variables are usually indexed by the dimension $d$ and we write $\todistr$ and $\toas$ in order to indicate, respectively, convergence in distribution and almost sure convergences, as $d\to\infty$.

Two sequences $(a_d)_{d\geq 1}$ and $(b_d)_{d\geq 1}$ of real numbers are called asymptotically equivalent, as $d\to\infty$, if
$\lim_{d\to\infty}\frac{a_d}{b_d}=1$. We denote this by writing $a_d\sim b_d$, as $d\to\infty$, or by $\underset{d\to\infty}{\sim}$ if there are several parameters tending to infinity and we want to emphasize the role of $d$. Furthermore, we use the usual Landau notation, that is we say that $a_d= o(b_d)$, as $d\to\infty$, if
\begin{align*}
\lim_{d\to\infty}\frac{a_d}{b_d}=0,
\end{align*}
and similarly, $a_d= O(b_d)$, as $d\to\infty$, if
\begin{align*}
\limsup_{d\to\infty}\left|\frac{a_d}{b_d}\right|<\infty.
\end{align*}
By slight abuse of notation we shall write $o(b_d)$ or $O(b_d)$ for the sequences $a_d$ themselves.

\subsection{Limit theorems for binomial random variables and binomial coefficients}\label{section:limit_theorem_binomial}

In this section, we collect some  limit theorems for binomial random variables, and thus, also for binomial coefficients. Besides the well-known central limit theorem and the law of large numbers for $\Bin(n,1/2)$, given by
\begin{align*}
\frac{2\Bin(n,1/2)-n}{\sqrt n}\overset{d}{\underset{n\to\infty}{\longrightarrow}} N(0,1)\qquad\text{and}\qquad \frac{\Bin(n,1/2)}{n}\overset{a.s.}{\underset{n\to\infty}{\longrightarrow}} 1,
\end{align*}
respectively, we will also need the local limit theorem. Following~\cite[VII., §1. Theorem~1]{Petrov1975}, it can be written  in the case of $\Bin(n,1/2)$ as follows:
\begin{align}\label{eq:local_limit_thm_asym_neu}
\sup_{m\in \ZZ}\left|\sqrt{\frac n4}\cdot\bP[\Bin(n,1/2)=m]-\frac{1}{\sqrt{2\pi}}\exp\left\{-\frac 12\Big(\frac{2m-n}{\sqrt{n}}\Big)^2\right\}\right|\underset{n\to\infty}{\longrightarrow} 0.
\end{align}
Additionally, we need an asymptotic expansion for the distribution function of the centered and normalized version of $\Bin(n,1/2)$. More precisely, for $Z_n:=(2\Bin(n,1/2)-n)/\sqrt n$, Theorem~IV.3 of~\cite{Esseen1945} states that
\begin{align}\label{eq:asym_binomial_esseen}
\bP[Z_n\le x]=\Phi(x)+\frac{\sqrt{2}}{\sqrt{\pi n}}\psi_n(x)e^{-x^2/2}+o\Big(\frac 1{\sqrt{n}}\Big),\qquad \text{as }n\to\infty,
\end{align}
uniformly in $x\in \RR$ (the uniformity in $x$ is discussed in Chapter V of \cite{Esseen1945}), where $\psi_n(x)$ is defined as
\begin{align}\label{eq:psi}
\psi_n(x):=Q\left(\frac{x\sqrt n}{2}-\frac n2\right),
\end{align}
for the $1$-periodic function $Q(x):=[x]-x+1/2$, where $[x]$ denotes the integer part of $x$.

The classical Cramer's theorem~\cite[Theorem I.3]{DenHollander} in the particular case of binomial random variables $\Bin(m,1/2)$  states that
\begin{equation}\label{eq:Cramer>1/2}
\lim_{m\to\infty}{1\over m}\log\bP[\Bin(m,1/2)\geq m(a+o(1))] = - \cI(a),\qquad a\in (1/2,1)
\end{equation}
and
\begin{equation}\label{eq:Cramer<1/2}
\lim_{m\to\infty}{1\over m}\log\bP[\Bin(m,1/2)\leq m(a+o(1))] = - \cI(a),\qquad a\in (0,1/2),
\end{equation}
where the \textit{information function} $\cI(a)$ is given by
\begin{align}\label{eq:Cramer_information}
\cI(a) =
%\begin{cases}
\log 2 + a\log a + (1-a)\log(1-a), \qquad  a\in[0,1].
%+\infty &: \text{otherwise}.
%\end{cases}
\end{align}
Here, $o(1)$ stands for any sequence tending to zero, as $m\to\infty$. Throughout this paper, the notation $\cI$ will be used solely for the information function as defined in~\eqref{eq:Cramer_information}.

We shall also need a precise version of the above Cramer asymptotics. It can be deduced as a special case from the general theory of mod-phi convergence~\cite{Feray_modphi2016}, in this case Theorem~3.2.2 from~\cite{Feray_modphi2016}.
%limit theorems for deviations for binomial random variables, and therefore, also for binomial coefficients.
Let $(x_m)_{m\ge 0}$ be a sequence such that $x_m\to x$ and $mx_m\in\NN$. Then, for all $x\in(0,1)$ it holds that
\begin{align}\label{eq:asym_binomial_=}
\bP[\Bin(m,1/2)=mx_m]\underset{m\to\infty}{\sim} \frac{e^{-m\cdot \cI(x_m)}}{\sqrt{2\pi m}}\frac{1}{\sqrt{x(1-x)}},
\end{align}
and for all $x\in(1/2,1)$, we have
\begin{align}\label{eq:asym_binomial_>=}
\bP[\Bin(m,1/2)\le m(1-x_m)]=\bP[\Bin(m,1/2)\ge mx_m]\underset{m\to\infty}{\sim} \frac{e^{-m\cdot \cI(x_m)}}{\sqrt{2\pi m}}\frac{1}{\sqrt{x(1-x)}}\frac{x}{2x-1}.
\end{align}
%Note that $\cI(x)$ denotes the information function from Cramer's theorem defined in~\eqref{eq:Cramer_information}.
The same asymptotic formulas also follow, e.g., from~\cite[Theorem~6]{Petrov1965}.

Moreover, we make use of the well-known Stirling's formula for $m!$, which says that
\begin{align*}
m!\underset{m\to\infty}{\sim} \sqrt{2\pi m}\Big(\frac me\Big)^m.
\end{align*}

\subsection{Convex cones, conic intrinsic volumes and quermassintegrals}\label{section:convex_geometry}

For a set $M\subset\RR^d$, the positive hull $\pos M$ of $M$ is defined as
\begin{align*}
\pos M:=\Big\{\sum_{i=1}^m\lambda_it_i:\,m\in\NN,t_1,\dots,t_m\in M,\lambda_1,\dots,\lambda_m\ge 0\Big\}.
\end{align*}
A convex set $C\subset\RR^d$ is called a \textit{convex cone}  if $\lambda C$ lies entirely in $C$ for all $\lambda\ge 0$.  In particular, $\pos M$ denotes the smallest convex cone containing $M$. In the following, we shall be interested in polyhedral cones, that is positive hulls of finite sets, and the word cone always refers to a polyhedral cone. A supporting hyperplane for a cone $C$ is a linear hyperplane $H$ such that $C$ lies entirely in one of the closed half-spaces bounded by $H$. Then, the cones $C\cap H$, where $H$ is a supporting hyperplane, are called \textit{faces} of the cone $C$. By convention, the cone $C$ is also a face of itself. A $k$-dimensional face of $C$ is also called \textit{$k$-face}, where the dimension of a face $F$ (and for a convex set in general) is defined as the dimension of its affine hull. The set of all $k$-faces of $C$ is denoted by $\cF_k(C)$, and the total number of distinct $k$-faces of $C$ is $f_k(C):=\#\cF_k(C)$. The \textit{dual cone $C^\circ$} of a cone $C\subset\RR^d$ is defined as
\begin{align*}
C^\circ:=\{v\in\RR^d:\langle v,x\rangle\le 0\;\forall x\in C\},
\end{align*}
where  $\langle\,\cdot\,,\,\cdot\,\rangle$ denotes the standard Euclidean scalar product.

Now, we briefly recall the definition of two geometric functionals of convex cones. For a cone $C\subset\RR^d$, the \textit{k-th conic intrinsic volume} $\upsilon_k(C)$ is defined as
\begin{align*}
\upsilon_k(C):=\sum_{F\in\cF_k(C)}\bP[\Pi_C(g)\in\relint\, F],\quad k\in\{0,\dots,d\}.
\end{align*}
Here, $g$ is a $d$-dimensional standard Gaussian distributed random vector and $\relint \,F$ denotes the relative interior of $F$, which is the set of all interior points of $F$ relative to its linear hull. Also, $\Pi_C(x)$, for a point $x\in\RR^d$, denotes the Euclidean projection of $x$ onto $C$, which is the unique point $y\in C$ minimizing the Euclidean distance to $x$. An equivalent definition of the conic intrinsic volumes using the spherical Steiner formula can be found in~\cite[Section~6.5]{SW}. For more background material and further properties we refer to~\cite[Section~2.2]{AmelunxenLotz} and~\cite[Section~6.5]{SW}.

For a cone $C\subset\RR^d$ that is not a linear subspace, the \textit{$k$-th conic quermassintegral} of $C$, for $k\in\{0,\dots,d\}$, is defined by
\begin{align*}
U_k(C):=\frac 12\bP[C\cap W_{d-k}\neq\{0\}],
\end{align*}
where $W_{d-k}$ is a random  $(d-k)$-subspace  uniformly distributed on the Grassmannian $G(d,d-k)$ of all $(d-k)$-dimensional linear subspaces of $\RR^d$ (the uniform distribution refers here to the unique rotation-invariant Haar probability measure).
For a $j$-dimensional linear subspace $L_j\subset\RR^d$, we define
\begin{align*}
U_k(L_j)=
\begin{cases}
1		&: j-k>0\text{ and odd},\\
0		&: j-k\le 0\text{ or even}.
\end{cases}
\end{align*}
Note that the values of $U_k$ for linear subspaces are chosen in such a way that the following linear relation between the conic intrinsic volumes and the conic quermassintegrals, called the \textit{conic Crofton formula}, holds for each cone $C\subset\RR^d$:
\begin{align}\label{eq:conic_Crofton}
U_k(C)=\upsilon_{k+1}(C)+\upsilon_{k+3}(C)+\ldots,
\end{align}
see, e.g.,~\cite[p.~262]{SW}.
If $C$ is not a linear subspace, the quantity $2U_k(C)$ is also called \textit{$k$-th Grassmann angle} of $C$ and was introduced by Grünbaum~\cite{gruenbaum_grass_angles}. For further properties of the conic quermassintegrals see also~\cite[Section~2]{HugSchneiderConicalTessellations}.

\section{Two models of random cones}\label{section:random_cones}

We formally define the two models of random cones mentioned in the introduction and collect some known properties. In this section, let $X_1,\dots,X_n$ be independent and $d$-dimensional random vectors distributed according to a common probability distribution $\phi$.  We assume that  $\phi$ is even (meaning that it is invariant under reflections at the origin) and  assigns measure zero to each $(d-1)$-dimensional hyperplane in $\RR^d$. Also, we always assume that $n\geq d$. Finally, we introduce the numbers
\begin{align}\label{eq:C(N,n)}
	C(n,d):=2\sum_{r=0}^{d-1}{n-1\choose r},
\end{align}
which will often appear in our considerations.

\subsection{The Donoho-Tanner random cones}

The \textit{Donoho-Tanner random cone} $D_{n,d}$ is defined to be the positive hull of the vectors $X_1,\dots,X_n$. Equivalently, $D_{n,d}$ can be described as the cone $A\RR^n_+$, where $A$ is the $d\times n$ random matrix with columns $X_1,\dots,X_n$ and $\RR_+^n=[0,\infty)^n$ is the positive orthant. It was shown by Donoho and Tanner in~\cite[Theorem 1.6]{DonohoTanner} that
\begin{equation}\label{eq:DonohoTannerF_k}
{\bE f_k(D_{n,d})\over{n\choose k}} = 1-{1\over 2^{n-k-1}}\sum_{r=0}^{n-d-1}{n-k-1\choose r}=1 - {1\over 2^{n-k}}\,C(n-k,n-d)
\end{equation}
for all $k\in \{1,\ldots,d-1\}$.
The dual cone of $D_{n,d}$, which is defined as
$$
D_{n,d}^\circ := \{x\in\RR^d:\langle x,y\rangle\leq 0\text{ for all }y\in D_{n,d}\},
$$
will be called the \textit{dual Donoho-Tanner random cone}. It can equivalently be described as
\begin{align*}
D_{n,d}^\circ= \{v\in\RR^d:\langle v,X_i\rangle\le 0\text{ for all } i=1,\dots,n\}.
\end{align*}

\subsection{The Schl\"afli and Cover-Efron random cones}

Consider  $n\in\NN$ independent linear random hyperplanes in $\RR^d$ whose unit normal vectors are $X_1,\dots,X_n$. By a formula of Schl\"afli~\cite[Lemma~8.2.1]{SW}, these hyperplanes dissect  $\RR^d$ into $C(n,d)$ non-degenerate cones, with probability $1$. The \textit{Schl\"afli random cone} $S_{n,d}$ is a cone chosen uniformly at random from these $C(n,d)$ cones. Hug and Schneider \cite{HugSchneiderConicalTessellations} determined the expected number of $k$-faces, the expected conic intrinsic volumes as well as the expected conic quermassintegrals of $S_{n,d}$. For later purposes we only recall here the formula for the expected conic intrinsic volumes of the Schl\"afli random cone $S_{n,d}$ from \cite[Corollary~4.3]{HugSchneiderConicalTessellations}:
\begin{align}\label{eq:intr_vol_Schlaefli}
\bE\upsilon_j(S_{n,d})=
\begin{cases}
\binom{n}{d-j}\frac{1}{C(n,d)}	&: j\in\{1,\dots,d\},\\
\binom{n-1}{d-1}\frac{1}{C(n,d)}&: j=0.
\end{cases}
\end{align}

The dual of the cone $S_{n,d}$, which is defined as
$$
S_{n,d}^\circ := \{x\in\RR^d:\langle x,y\rangle\leq 0\text{ for all }y\in S_{n,d}\},
$$
is called the \textit{Cover-Efron random cone} and will be denoted by $C_{n,d}$. It can equivalently be described as the positive hull of $n$ independent random vectors $X_1,\ldots,X_n$ with distribution $\phi$, conditioned on the event that this cone is not the whole space $\RR^d$. By Wendel's formula~\cite{Wendel} (see also~\cite[Theorem~8.2.1]{SW}  and~\cite{mycielski}), the probability of the event $\{D_{n,d}\neq \RR^d\}$ is $C(n,d)/2^n$.  That is, $C_{n,d}$ is the random cone with distribution given by $\bP[C_{n,d}=\RR^d]=0$ and
$$
\bP[C_{n,d}\in B]
=
\frac {2^n}{C(n,d)}\int_{(\SS^{d-1})^n}{\bf 1}\{\pos(x_1,\ldots,x_n)\in B\}\,\phi^n(\dint(x_1,\ldots,x_n))
$$
for each Borel set $B$ in the space of cones that does not contain the element $\RR^d$.
Cover and Efron \cite{CoverEfron} have shown that
\begin{equation}\label{eq:CoverEfronF_k}
\bE f_k(C_{n,d}) = {2^k{n\choose k}C(n-k,d-k)\over C(n,d)}
\end{equation}
holds for any $k\in\{0,\ldots,d-1\}$, see also \cite[Corollary 4.1]{HugSchneiderConicalTessellations}. From \cite[Corollaries~4.2,~4.3]{HugSchneiderConicalTessellations} it is also known that
\begin{align}\label{eq:intr_vol_cover-efron}
\bE v_k(C_{n,d}) = \begin{cases}
{{n\choose k}\over C(n,d)} &: k\in\{0,1,\ldots,d-1\}\\
{{n-1\choose d-1}\over C(n,d)} &: k=d.
\end{cases}
\end{align}
and
\begin{align}\label{eq:quermassint_cover-efron}
\bE U_k(C_{n,d})= \frac{C(n,d)-C(n,k)}{2C(n,d)},\quad k\in\{1,\dots,d\}.
\end{align}

It is important to note that the two constructions of random cones we described are very similar in the following sense. The Cover-Efron random cone $C_{n,d}$ has the same distribution as the Donoho-Tanner random cone $D_{n,d}$ conditioned on the event that $D_{n,d}\neq\RR^d$. Similarly, $S_{n,d}$ has the same distribution as $D_{n,d}^\circ$ conditioned on the event that $D_{n,d}^\circ\neq\{0\}$.

\section{Limit theorems for the expected face numbers}\label{section:limit_faces}

In this section, our goal is to understand the asymptotic behaviour of the expected number of $k$-faces of the random cones introduced in Section~\ref{section:random_cones} in high dimensions, meaning that the number of vectors $n$, the dimension $d$, and in some cases also $k$ tend to infinity simultaneously and in a coordinated way.

\subsection{Fixed face-dimension \texorpdfstring{$\boldsymbol{k}$}{k}}

We start by considering the case where the face-dimension $k\in\NN$ is fixed.
Let $C_{n,d}$ be a Cover-Efron random cone. In Theorem $1.1$ of~\cite{HugSchneiderThresholdPhenomena}, Hug and Schneider proved that if $n=n(d)$ is such that
\begin{align*}
\frac dn \to \delta,\qquad\text{as }d\to\infty,
\end{align*}
for a parameter $\delta\in[0,1]$, the Cover-Efron random cone satisfies the following threshold phenomenon:
\begin{align*}
\lim_{d\to\infty}\frac{\bE f_k(C_{n,d})}{\binom nk}=
\begin{cases}
1		&:\delta\in(1/2,1],\\
(2\delta)^k		&:\delta\in[0,1/2).
\end{cases}
\end{align*}
In this section, we are going to further evaluate what happens in the critical case $\delta=1/2$.
In addition, in their Theorem 1.2, Hug and Schneider~\cite{HugSchneiderThresholdPhenomena} proved that if $n=n(d)$ is such that $n-2d$ stays bounded from above, as $d\to\infty$,  then for every fixed $k\in \NN$, one has that
$$
\lim_{d\to\infty} \frac{\bE f_k(C_{n,d})}{\binom n k} = 1.
$$ 
In fact, Theorem~1.2 in~\cite{HugSchneiderThresholdPhenomena} gives the same conclusion under weaker but more technical conditions.
Hug and Schneider~\cite[p.~567]{HugSchneiderThresholdPhenomena} asked what happens under the weaker condition $n\sim 2d$. The next theorem addresses this question. 

\begin{theorem}\label{theo:4.2}
Let $C_{n,d}$ be a Cover-Efron random cone. Suppose that $n=n(d)$ is such that
$$
\frac dn\to \frac 12,\qquad\text{as $d\to\infty$}.
$$
Then, for every fixed $k\in \NN$, it holds that
$$
\lim_{d\to\infty} \frac{\bE f_k(C_{n,d})}{\binom n k} = 1.
$$
\end{theorem}

\begin{proof}
Let us first introduce the quantity
$$
Q_k(n,d) := \frac{\bP[\Bin(n-k,1/2) \leq d-k]}{\bP[\Bin (n,1/2) \leq d]}.
$$
By~\eqref{eq:CoverEfronF_k}, we can write $\bE f_k(C_{n,d})$ in terms of binomial probabilities:
\begin{align}\label{eq:k-faces_covefron_as_binomial_probs}
\frac{\bE f_k(C_{n,d})}{\binom nk}=\frac{\bP[\Bin(n-k-1,1/2) \leq d-k-1]}{\bP[\Bin (n-1,1/2) \leq d-1]}.
\end{align}
Assuming that $d/n\to 1/2$ our task is to prove that
$$
\lim_{d\to\infty} Q_k(n,d) = 1,
$$
since we have $\bE f_k(C_{n,d})/\binom nk = Q_k(n-1,d-1)$ by~\eqref{eq:k-faces_covefron_as_binomial_probs}, and the condition $n\sim 2d$ implies that $n-1\sim 2(d-1)$.

The main difficulty is that the condition $n\sim 2d$ is too general and leaves open several possibilities for the asymptotic behaviour of the denominator and the numerator in the definition of $Q_k(n,d)$. More precisely, the limit $c:=\lim_{d\to\infty} (\frac n2 - d)/\sqrt n$ may be finite (in which case the denominator converges to $\Phi(-2c)$), it may be $-\infty$ (in which case the denominator converges to $1$) or it may be $+\infty$ (in which case the denominator converges to $0$), or it may not exist at all.

To determine the asymptotic behaviour of the denominator and the numerator in the case when both converge to $0$, we shall use the the Cram\'er-Petrov theorem on moderate deviations obtained by Cram\'er~\cite{cramer} and subsequently strengthened in~\cite{feller,hoeglund,petrov_cramer}; see also~\cite[Chapter VIII, \S~2]{Petrov1975} and~\cite[Theorem~8.1.1]{ibragimov_linnik_book}. The theorem is valid for sums of i.i.d.\ random variables. Here, we use it in the special case of the binomial distribution.
To state the asymptotics, let  $y_n>0$ be a sequence such that $y_n/\sqrt n\to\infty$ and $y_n/n\to 0$. Then,  we have
\begin{equation}\label{eq:cramer_petrov_asymptotics}
\bP\left[\Bin (n,1/2) \leq \frac{n}{2} - y_n\right] \sim \frac{\sqrt n}{\sqrt{2\pi} y_n} \exp\left\{-n \cdot\fI\left(\frac{y_n}n\right)\right\},
\qquad
\text{ as } n\to\infty,
\end{equation}
where the information function $\fI(z):[-\frac 12, +\frac 12] \to [0,\infty)$ is given by
$$
\fI(z) = \log 2 + \left(z+\frac 12\right)\log \left(z+\frac 12\right) + \left(\frac 12 -z\right)\log\left(\frac 12 - z\right).
$$
We define the sequences
\begin{equation}\label{eq:y_N_z_N_def}
y_n:= \frac n2 - d, \qquad z_n:=\frac n2 - d + \frac k2.
\end{equation}
Then, our assumption $n\sim 2d$ implies that $y_n/n \to 0$, but in general we need not have $y_n/\sqrt n \to +\infty$.  Therefore, we consider several cases.

\vspace*{2mm}
\noindent
\textsc{Case 1.}
Assume first that $y_n/\sqrt n\to c \in \RR\cup\{-\infty\}$ as $n\to\infty$. Then, the central limit theorem for binomial random variables, applied twice, yields
$$
Q_k(n,d)
=
\frac{\bP\left[\Bin(n-k,1/2) \leq \frac{n-k}{2} - z_n\right]}{\bP\left[\Bin (n,1/2) \leq \frac n2 - y_n\right]}
\to
\frac{\Phi(-2c)}{\Phi(-2c)} = 1.
$$
Note that the above argument applies if $c=-\infty$, in which case both the denominator and the numerator converge to $\Phi(+\infty)=1$. Observe that  in the case $c=+\infty$ (which we excluded above), the argument would break down because it would lead to an indeterminacy $0/0$. In the next step we shall resolve the indeterminacy by using the Cram\'er-Petrov asymptotics.

\vspace*{2mm}
\noindent
\textsc{Case 2.}
Assume now that $y_n/\sqrt n \to +\infty$, which also implies that $z_n/\sqrt n\to+\infty$ and $z_n = y_n + \frac k2 \sim y_n$. Applying the Cram\'er-Petrov asymptotics~\eqref{eq:cramer_petrov_asymptotics} twice, we have
\begin{align*}
Q_k(n,d)
=
\frac{\bP\left[\Bin(n-k,1/2) \leq \frac{n-k}{2} - z_n\right]}{\bP\left[\Bin (n,1/2) \leq \frac n2 - y_n\right]}
\sim
\frac
{\frac{\sqrt{n-k}}{\sqrt {2\pi}z_n} \exp\left\{- (n-k)\cdot \fI \left(\frac{z_n}{n-k}\right)\right\}}
{\frac{\sqrt n}{\sqrt {2\pi} y_n} \exp\left\{-n \cdot\fI\left(\frac {y_n}{n}\right)\right \}}.
\end{align*}
It follows from $y_n\sim z_n$ and $n-k\sim n$ that
$$
Q_k(n,d)
\sim
\exp\left\{n\cdot \fI \left(\frac {y_n}{n}\right) - (n-k)\cdot \fI \left(\frac{z_n}{n-k}\right)\right\}.
$$
The condition $n\sim 2d$ implies that $y_n = o(n)$ and $z_n= o(n)$, hence $z_n/(n-k) \to 0$ and $k \fI(z_n/(n-k)) \to 0$. It follows that the above result can be simplified to
$$
Q_k(n,d)
\sim
\exp\left\{n \fI \left(\frac {y_n}{n}\right) - n \fI \left(\frac{z_n}{n-k}\right)\right\}.
$$
By~\eqref{eq:y_N_z_N_def} we have
$$
\frac{z_n}{n-k}
=
\frac{y_n}{n} \cdot \frac{1}{1- \frac kn}  + \frac{k}{2(n-k)}
=
\frac {y_n}{n}\left(1+ O\left(\frac{1}{n}\right)\right) + O\left(\frac1n \right)
%=
%\frac {y_n}{n} + O\left(\frac{y_n}{n^2}\right) + O\left(\frac1n \right)
=
\frac {y_n}{n}  + O\left(\frac1n \right)
$$
since $y_n/n \to 0$. By the mean value theorem, for some sequence $\xi_n\to 0$ sandwiched between $y_n/n$ and $z_n/(n-k)$, we have
$$
\fI \left(\frac {y_n}{n}\right) -  \fI \left(\frac{z_n}{n-k}\right)
=
O\left(\frac 1n\right) \fI'(\xi_n)
=
O\left(\frac {\xi_n}{n} \right)
=
O\left(\frac {y_n}{n^2} \right),
$$
where we also used that  $\fI'(\xi_n) = O(\xi_n)$ by the Taylor expansion of $\fI$ at $0$ and then that  $y_n/n$ and $z_n/(n-k)$ are asymptotically equivalent implying that $\xi_n \sim y_n/n$.
Altogether, it follows that
$$
Q_k(n,d) = (1+o(1)) \exp\left\{n O\left(\frac {y_n}{n^2}\right)\right\} = (1+o(1)) \exp\left\{O\left(\frac {y_n}{n}\right)\right\} \to 1,
$$
because $y_n/n\to 0$. This completes the proof under the assumption $y_n/\sqrt n \to +\infty$.

\vspace*{2mm}
\noindent
\textsc{Case 3.}
In the previous two steps we proved our claim in the case when the limit of $y_n/\sqrt n$ exists in $\RR\cup \{+\infty, -\infty\}$. In general, the limit need not exist. To treat this case, we use the subsequence argument. Assume by contraposition that the limit of $Q_k(n,d)$ over some subsequence of $d$'s is not equal to $1$ or does not exist. Extracting a further subsequence, we may assume that $y_n/\sqrt n$ has a limit in $\RR\cup \{+\infty, -\infty\}$. Then, we apply either Step 1 or Step 2, implying that over this subsequence $Q_{d}(n,d)$ converges to $1$, which is a contradiction.
\end{proof}

In Theorem~1.3 of~\cite{HugSchneiderThresholdPhenomena} Hug and Schneider also proved that if $n=2d$, then for all fixed $k\in \NN$, one has that
\begin{align*}
\lim_{d\to\infty} \sqrt{d}\left(1-\frac{\bE f_k(C_{n,d})}{\binom nk}\right)=\frac{k}{\sqrt{\pi}}.
\end{align*}
The next theorem refines this result by looking more closely into the critical window.

\begin{theorem}\label{theorem:f_k_CoverEfron_k_fixed}
Let $C_{n,d}$ be a Cover-Efron random cone. Suppose that $n=n(d)$ is such that
$$
n=2d+c\sqrt{d}+o(\sqrt{d}),\qquad\text{as }d\to\infty,
$$
for a parameter $c\in\RR$.
Then, it holds that
\begin{align*}
\lim_{d\to\infty}\sqrt{d}\left(1-\frac{\bE f_k(C_{n,d})}{\binom nk}\right)=\frac{e^{-c^2/4}}{\Phi(-c/\sqrt 2)}\cdot\frac k{2\sqrt{\pi}}
\end{align*}
for all fixed $k\in\NN$.
\end{theorem}

\begin{proof}
Let $n=2d+c\sqrt{d}+o(\sqrt{d})$ as $d\to\infty$. In view of~\eqref{eq:k-faces_covefron_as_binomial_probs}, we obtain
\begin{align}\label{eq:1-f_k_binomial}
1-{\bE f_k(C_{n,d})\over{n\choose k}}=\frac{\bP[\Bin(n-1,1/2)\le d-1]-\bP[\Bin(n-k-1,1/2)\le d-k-1]}{\bP[\Bin(n-1,1/2)\le d-1]}.
\end{align}
Defining $Z_n:=(2\Bin(n,1/2)-n)/\sqrt n$ and using the central limit theorem for the binomial distribution, we find the following convergence for the denominator:
\begin{align*}
\bP[\Bin(n-1,1/2)\le d-1]
&	=\bP\left[Z_{n-1}\le \frac{2(d-1)-(n-1)}{\sqrt{n-1}}\right]\\
%	=\bP\left[Z_{n-1}\le \frac{-c\sqrt{d}+o(\sqrt{d})}{\sqrt{n-1}}\right]\\
&	=\bP\left[Z_{n-1}\le -\frac{c}{\sqrt 2}+o(1)\right]\ton \Phi\Big(-\frac{c}{\sqrt 2}\Big).
\end{align*}
In order to obtain the asymptotic behaviour of the numerator of~\eqref{eq:1-f_k_binomial}, we define the sequences $(c_n)_{n\ge 0}$ and $(d_n)_{n\ge 0}$ by
\begin{align*}
c_n:=\frac{2d-n-1}{\sqrt{n-1}}\qquad\text{and}\qquad d_n:=\frac{2d-k-n-1}{\sqrt{n-k-1}}.
\end{align*}
Note that both sequences converge to $-c/\sqrt 2$ as $d\to\infty$.
Then, we can use the asymptotic expansion~\eqref{eq:asym_binomial_esseen}  to obtain
\begin{align*}
&\bP[\Bin(n-1,1/2)\le d-1]-\bP[\Bin(n-k-1,1/2)\le d-k-1]\\
&	\quad=\bP[Z_{n-1}\le c_n]-\bP[Z_{n-k-1}\le d_n]\\
&	\quad=\Phi(c_n)-\Phi(d_n)+\frac{\sqrt 2}{\sqrt{\pi(n-1)}}\psi_{n-1}(c_n)e^{-c^2/4}-\frac{\sqrt 2}{\sqrt{\pi(n-k-1)}}\psi_{n-k-1}(d_n)e^{-c^2/4}+o\Big(\frac 1{\sqrt n}\Big).
\end{align*}
Inserting $c_n$ and $d_n$ into the definitions of $\psi_{n-1}$ and $\psi_{n-k-1}$, respectively (see~\eqref{eq:psi}), yields $\psi_{n-1}(c_n)=\frac 12=\psi_{n-k-1}(d_n)$, and thus
\begin{equation}\label{eq:Difference}
\begin{split}
&\bP[\Bin(n-1,1/2)\le d-1]-\bP[\Bin(n-k-1,1/2)\le d-k-1]\\
&\hspace{7cm} =\Phi(c_n)-\Phi(d_n)+o\Big(\frac 1{\sqrt n}\Big).
\end{split}
\end{equation}
Using the mean value theorem for the continuous and differentiable function $\Phi$ yields that we can find a sequence $(\xi_n)_{n\ge 0}$ such that $\xi_n$ lies between $c_n$ and $d_n$, and  satisfies
\begin{align*}
\Phi(c_n)-\Phi(d_n)=(c_n-d_n)\Phi'(\xi_n).
\end{align*}
Since both sequences $(c_n)_{n\ge 0}$ and $(d_n)_{n\ge 0}$ converge to $-c/\sqrt 2$, as $d \to\infty$, we have that
\begin{align*}
\lim_{n\to\infty}\Phi'(\xi_n)=\Phi'\Big(-\frac{c}{\sqrt 2}\Big)=\frac{1}{\sqrt{2\pi}}e^{-c^2/4}.
\end{align*}
Hence, it is left to determine the asymptotics of the difference $c_n-d_n$. We can write
\begin{align*}
c_n=\frac 1{\sqrt{n}} \cdot\frac{2d-n-1}{\sqrt{1- \frac 1n}}\qquad\text{and}\qquad d_n=\frac 1{\sqrt{n}} \cdot\frac{2d-k-n-1}{\sqrt{1- \frac{k-1}{n}}}.
\end{align*}
Using that $(1+x)^a=1+ax+o(x)$ as $x\to 0$, we obtain
\begin{align*}
\frac{1}{\sqrt{1-\frac 1n}}=1+\frac 1{2n}+o\Big(\frac 1n\Big)\qquad\text{and}\qquad \frac{1}{\sqrt{1-\frac {k-1}n}}=1+\frac {k-1}{2n}+o\Big(\frac 1n\Big).
\end{align*}
This yields
\begin{align*}
c_n=\frac{2d-n-1}{\sqrt{n}}+\frac{2d-n-1}{2n^{3/2}}+o\Big(\frac{1}{\sqrt n}\Big)
\end{align*}
and
\begin{align*}
d_n=\frac{2d-n-1}{\sqrt{n}}-\frac{k}{\sqrt n}+\frac{(2d-k-n-1)(k-1)}{2n^{3/2}}+o\Big(\frac{1}{\sqrt n}\Big).
\end{align*}
Thus, we obtain the following expansion for the difference of $c_n$ and $d_n$:
\begin{align*}
c_n-d_n
%&	=\frac{k}{\sqrt n}+\frac{2d-n-1}{2n^{3/2}}-\frac{(2d-k-n-1)(k-1)}{2n^{3/2}}+o\Big(\frac{1}{\sqrt n}\Big)\\
%&	=\frac{k}{\sqrt n}+\frac{(k-2)(n-2d)}{2n^{3/2}}+o\Big(\frac{1}{\sqrt n}\Big)\\
&	=\frac{k}{\sqrt n}+o\Big(\frac{1}{\sqrt n}\Big).
\end{align*}
Inserting the results into~\eqref{eq:Difference} yields
\begin{align*}
\bP[\Bin(n-1,1/2)\le d-1]-\bP[\Bin(n-k-1,1/2)\le d-k-1]= \frac{k}{\sqrt{2\pi n}}e^{-c^2/4}+o\Big(\frac{1}{\sqrt n}\Big).
\end{align*}
Combining this with~\eqref{eq:1-f_k_binomial} leads to
\begin{align*}
\sqrt{d}\left(1-\frac{\bE f_k(C_{n,d})}{\binom nk}\right)\tosim \sqrt{d}\cdot\frac{k}{\sqrt{2 \pi n}}e^{-c^2/4}\cdot\frac{1}{\Phi(-c/\sqrt 2)}.
\end{align*}
Finally, the fact that $\sqrt{d}/\sqrt{n}\to 1/\sqrt 2$ proves the claim.
\end{proof}

The case of the Donoho-Tanner random cone $D_{n,d}$ is much simpler.
The next result is similar to~\cite[Theorem~1.1]{HugSchneiderThresholdPhenomena}, but deals with $\bE f_k(D_{n,d})$, and is essentially present in~\cite{DonohoTanner}. We state it for completeness.

\begin{proposition}\label{prop:DT61220}
Let $D_{n,d}$ be a Donoho-Tanner random cone. Suppose that  $n=n(d)$ is such that
$$
\frac dn\to\delta,\qquad\text{as }d\to\infty,
$$
for a parameter $\delta\in[0,1]$. Then, it holds that
\begin{align*}
\lim_{d\to\infty}\frac{\bE f_k(D_{n,d})}{\binom nk}=
\begin{cases}
1			&:\delta\in(1/2,1],\\
0			&:\delta\in[0,1/2).
\end{cases}
\end{align*}
In the critical case where
$$
n=2d+c\sqrt{d}+o(\sqrt d),\qquad\text{as }d\to\infty,
$$
for some parameter $c\in\RR$, we have
$$
\lim_{d\to\infty}\frac{\bE f_k(D_{n,d})}{\binom nk}=\Phi\Big(-\frac{c}{\sqrt{2}}\Big).
$$
\end{proposition}

\begin{proof}
By~\eqref{eq:DonohoTannerF_k}, we can write $\bE f_k(D_{n,d})/\binom nk$ in terms of binomial probabilities:
\begin{align}\label{eq:f_k-Donoho_binomial}
\frac{\bE f_k(D_{n,d})}{\binom nk}=1-\bP[\Bin(n-k-1,1/2)<n-d]=\bP[\Bin(n-k-1,1/2)\ge n-d].
\end{align}
Since $(n-d)/(n-k-1)=(1-\delta)+o(1)$, we can use the law of large numbers to obtain
\begin{align*}
\lim_{d\to\infty}\frac{\bE f_k(D{n,d})}{\binom nk}=\bP[\Bin(n-k-1,1/2)\ge (1-\delta+o(1))(n-k-1)]=
\begin{cases}
1			&:\delta\in(1/2,1],\\
0			&:\delta\in[0,1/2),
\end{cases}
\end{align*}
which proves the first claim.

In the critical case $n=2d+c\sqrt d+o(\sqrt d)$, for a parameter $c\in\RR$, we can define $Z_n:=(2\Bin(n,1/2)-n)/\sqrt n$ and use the central limit theorem for binomial random variables to get
\begin{align*}
\frac{\bE f_k(D_{n,d})}{\binom nk}
&	=\bP\left[Z_{n-k-1}\ge \frac{2(n-d)-(n-k-1)}{\sqrt{n-k-1}}\right]\\
&	=\bP\left[Z_{n-k-1}\ge \frac{c}{\sqrt 2}+o(1)\right]\ton \Phi\Big(-\frac{c}{\sqrt 2}\Big),
\end{align*}
which completes the proof.

\end{proof}

\subsection{Increasing face-dimension \texorpdfstring{$\boldsymbol{k}$}{k}}

In this section our goal is to understand the asymptotic behaviour of the expected face numbers when not only $n=n(d)$ but also $k=k(d)$ is a function of $d$ and all parameters tend to infinity in a  coordinated way. The next theorem extends and considerably refines Theorem~1.7 of~\cite{HugSchneiderThresholdPhenomena}.

\begin{theorem}\label{theorem:limit_theorem_faces_cover-efron}
Let $C_{n,d}$ be a Cover-Efron random cone. Suppose that $k=k(d)$ and $n=n(d)$ are such that
\begin{align}\label{eq:regime_Cover-Efron}
\frac dn \to \delta\quad\text{and}\quad \frac kd \to\varrho,\qquad \text{as }d\to\infty,
\end{align}
for $\delta\in(0,1)$ and $\varrho\in(0,1)$. Then, it holds that
%\TG[{\color{red}Referee: Change cases to $\varrho<2-1/\delta$ and $\varrho>2-1/\delta$??}]
\begin{align}\label{eq:first_part_hug_schneider}
\lim_{d\to\infty}\frac{\bE f_k(C_{n,d})}{\binom nk}=
\begin{cases}
1		&: \delta>1/2\text{ and }\varrho<2-1/\delta,\\
0		&: \delta>1/2\text{ and }\varrho>2-1/\delta,\\
0		&: \delta\le 1/2.
\end{cases}
\end{align}
Furthermore, if $\delta>1/2$ and
\begin{align}\label{eq:regime_cover_critical1}
n=\frac d\delta+c\sqrt d +o(\sqrt d)\quad\text{and}\quad k=\Big(2-\frac 1\delta\Big)d+b\sqrt{d}+o(\sqrt{d}),\qquad\text{as }d\to\infty,
\end{align}
for parameters $c,b\in\RR$, then it holds that
\begin{align*}
\lim_{d\to\infty}\frac{\bE f_k(C_{n,d})}{\binom nk}=\Phi\bigg(-\frac{c+b}{\sqrt{2/\delta-2}}\bigg).
\end{align*}
In the case where
\begin{align}\label{eq:regime_cover_critical2}
n=2d+c\sqrt{d}+o(\sqrt{d})\quad\text{and}\quad k=b\sqrt{d}+o(\sqrt{d}),\qquad\text{as }d\to\infty,
\end{align}
for constants $c\in\RR$ and $b\geq 0$, it holds that
\begin{align*}
\lim_{d\to\infty}\frac{\bE f_k(C_{n,d})}{\binom nk}=\frac{\Phi(-(c+b)/\sqrt{2})}{\Phi(-c/\sqrt{2})}.
\end{align*}
\end{theorem}

\begin{remark}
	\begin{itemize}
		\item[(i)] Note that an analogous result to~\eqref{eq:first_part_hug_schneider}  for the Donoho-Tanner random cone was proven in~\cite[Theorem~1.2]{DonohoTanner}.
		\item[(ii)] Hug and Schneider~\cite{HugSchneiderThresholdPhenomena} proved~\eqref{eq:first_part_hug_schneider} using various inequalities for the binomial coefficients.
		We provide a different, more conceptual and probabilistic,  proof using the representation of $\bE f_k(C_{n,d})$ as binomial probabilities, which allows us to build on probabilistic limit theorems.
	\end{itemize}
\end{remark}

\begin{proof}[Proof of Theorem \ref{theorem:limit_theorem_faces_cover-efron}]
%The starting point is a representation of $\bE f_k(C_{n,n})$ in terms of binomial probabilities. In fact, from~\eqref{eq:CoverEfronF_k} together with the definition of binomial random variables it follows that
%\begin{align}\label{eq:k-faces_covefron_as_binomial_probs}
%{\bE f_k(C_{n,n})\over{n\choose k}} = {\bP[\Bin(n-k-1,1/2)\le n-k-1]\over\bP[\Bin(n-1,1/2)\leq n-1]}.
%\end{align}
Recall from~\eqref{eq:k-faces_covefron_as_binomial_probs} that
\begin{align}\label{eq:k-faces_covefron_as_binomial_probs1}
\frac{\bE f_k(C_{n,d})}{\binom nd}=\frac{\bP[\Bin(n-k-1,1/2) \leq d-k-1]}{\bP[\Bin (n-1,1/2) \leq d-1]}.
\end{align}
Assume that~\eqref{eq:regime_Cover-Efron} holds true. We start with the case $\delta>1/2$. The law of large numbers implies that
$$
\lim_{d\to\infty} \bP[\Bin(n-1,1/2)\leq d-1] = 1.
$$
Putting
\begin{align*}
a:=\frac{1-\varrho}{1/\delta-\varrho},
\end{align*}
we observe that
\begin{align*}
\frac{d-k-1}{n-k-1}=a+o(1).
\end{align*}
Since $a>1/2$ is equivalent to $\varrho<2-1/\delta$, we obtain that
\begin{align*}
&\lim_{d\to\infty}\bP[\Bin(n-k-1,1/2)\le d-k-1]\\
&	\qquad=\lim_{d\to\infty}\bP[\Bin(n-k-1,1/2)\le (a+o(1))(n-k-1)]\\
&	\qquad=\begin{cases}
	1		&: \varrho<2-1/\delta,\\
	0		&: \varrho>2-1/\delta,
	\end{cases}
\end{align*}
by the law of large numbers. Inserting both limits into~\eqref{eq:k-faces_covefron_as_binomial_probs1} proves the case $\delta >1/2$.

For $\delta\le 1/2$, we need to show that $\lim_{d\to\infty}\bE f_k(C_{n,d}) /\binom{n}{k}=0$. At first consider the case $\delta\in(0, 1/2)$. Then, $a\in(0,1/2)$ and we can apply the asymptotic equivalence~\eqref{eq:asym_binomial_>=} in both the numerator and the denominator of~\eqref{eq:k-faces_covefron_as_binomial_probs1}. This  yields
\begin{align*}
\frac{\bE f_k(C_{n,d})}{\binom nk}
& 	\tosim C(\delta,\varrho)\cdot\frac{\exp\{-(n-k-1)\cdot\cI(\frac{d-k-1}{n-k-1})\}}{\exp\{-(n-1)\cdot\cI(\frac{d-1}{n-1})\}}\\
&	\tosim C(\delta,\varrho)\cdot\exp\left\{-n\left(\Big(\frac{n-k-1}{n}\Big)\cI\Big(\frac{d-k-1}{n-k-1}\Big)-\Big(\frac{n-1}{n}\Big)\cI\Big(\frac{d-1}{n-1}\Big)\right)\right\}\\
&	\tosim C(\delta,\varrho) \cdot\exp\left\{-n\left((1-\delta\varrho)\cI(a)-\cI(\delta)+o(1)\right)\right\},
\end{align*}
where $C(\delta,\varrho)>0$ is a constant which only depends on $\delta$ and $\varrho$. Note that $\cI$ again denotes the information function~\eqref{eq:Cramer_information}. Thus, it suffices to prove that
$$
(1-\varrho\delta)\cI(a)-\cI(\delta)>0,
$$
or equivalently, that
$$
\frac{\cI(a)}{\cI(\delta)}>\frac{1}{1-\varrho\delta}.
$$
Using the convexity of the information function $\cI$ and the fact that $a<\delta< 1/2$ and $\cI(1/2)=0$, it follows  that
\begin{align*}
\cI(a)\ge \cI(\delta)\frac{1/2 -a}{1/2 -\delta}.
\end{align*}
Thus, it remains to show that
\begin{align*}
\frac{1/2 -a}{1/2 -\delta}-\frac{1}{1-\varrho\delta}>0.
\end{align*}
With some elementary transformations, we obtain
\begin{align*}
\frac{1/2 -a}{1/2 -\delta} - \frac{1}{1-\varrho\delta} =\frac{1+\varrho\delta-2\delta}{(1-\varrho\delta)(1-2\delta)}-\frac{1}{1-\varrho\delta}
=\frac{\varrho\delta}{(1-\varrho\delta)(1-2\delta)}>0,
\end{align*}
%and consequently,
%\begin{align*}
%\frac{1/2 -a}{1/2 -\delta}-=\frac{\varrho\delta}{(1-\varrho\delta)(1-2\delta)}>0,
%\end{align*}
which completes the proof of the case $\delta\in(0,1/2)$.

In the case $\delta=1/2$, take an $\eps\in(0,1)$ and use~\eqref{eq:k-faces_covefron_as_binomial_probs1} to obtain
\begin{align*}
\frac{\bE f_k(C_{n,d})}{{n\choose k}}
&	= \frac{\bP[\Bin(n-k-1,1/2)\le \frac{d-k-1}{n-k-1}(n-k-1)]}{\bP[\Bin(n-1,1/2)\leq \frac{d-1}{n-1}(n-1)]}\\
&	\le \frac{\bP[\Bin(n-k-1,1/2)\le \frac{d-k-1}{n-k-1}(n-k-1)]}{\bP[\Bin(n-1,1/2)\leq \frac{(1-\eps)d-1}{n-1}(n-1)]}.
\end{align*}
Again, applying~\eqref{eq:asym_binomial_>=} this yields
\begin{align*}
&\frac{\bP[\Bin(n-k-1,1/2)\le \frac{d-k-1}{n-k-1}(n-k-1)]}{\bP[\Bin(n-1,1/2)\leq \frac{(1-\eps)d-1}{n-1}(n-1)]} \\
&	\quad\tosim D(\varrho,\eps)\cdot\exp\left\{-n\left(\Big(\frac{n-k-1}{n}\Big)\cI\Big(\frac{d-k-1}{n-k-1}\Big)-\Big(\frac{n-1}{n}\Big)\cI\Big(\frac{(1-\eps)d-1}{n-1}\Big) \right)\right\}\\
&	\quad\tosim  D(\varrho,\eps)\cdot\exp\left\{-n\left(\Big(1-\frac{\varrho}{2}\Big)\cI(a)-\cI\Big(\frac{1-\eps}{2}\Big)+o(1)\right)\right\},
\end{align*}
where $D(\varrho,\eps)$ is some constant which only depends on $\varrho$ and $\eps$. Since $a<\delta = 1/2$, we can choose $\eps>0$ sufficiently close to zero, such that
\begin{align*}
\Big(1-\frac{\varrho}{2}\Big)\cI(a)>\cI\Big(\frac{1-\eps}{2}\Big).
\end{align*}
This proves the claim.

Now, consider the regime~\eqref{eq:regime_cover_critical1}. Since $\delta>1/2$, it holds that
$$
\lim_{d\to\infty} \bP[\Bin(n-1,1/2)\leq d-1] = 1.
$$
Defining $Z_n:=(2\Bin(n,1/2)-n)/\sqrt{n}$, yields
\begin{align*}
\bP[\Bin(n-k-1,1/2)\le d-k-1]
&	=\bP\bigg[Z_{n-k-1}\le \frac{2(d-k-1)-(n-k-1)}{\sqrt{n-k-1}}\bigg].
\end{align*}
Since
\begin{align*}
n-k-1=\Big(\frac 2\delta-2\Big)d+(c-b)\sqrt{d}+o(\sqrt{d})
\end{align*}
and
\begin{align*}
2(d-k-1)-(n-k-1)=-(c+b)\sqrt{d}+o(\sqrt{d}),
\end{align*}
we have
\begin{align*}
\bP[\Bin(n-k-1,1/2)\le d-k-1]
%&	=\bP\bigg[Z_{n-k-1}\le \frac{-b\sqrt{d}+o(\sqrt{d})}{\sqrt{(2/\delta-2)d-b\sqrt{d}+o(\sqrt{d})}}\bigg]\\
	=\bP\bigg[Z_{n-k-1}\le -\frac {c+b}{\sqrt{2/\delta-2}}+o(1)\Big]
	\ton \Phi\bigg(-\frac{c+b}{\sqrt{2/\delta - 2}}\bigg),
\end{align*}
following from the central limit theorem for the binomial distribution. Combining this with~\eqref{eq:k-faces_covefron_as_binomial_probs1} yields the claim.

Finally, consider the regime~\eqref{eq:regime_cover_critical2}. Similarly to the previous case, we obtain
\begin{align*}
\bP[\Bin(n-1,1/2)\le d-1]
&	=\bP\left[Z_{n-1}\le \frac{2(d-1)-(n-1)}{\sqrt{n-1}}\right]\\
%&	=\bP\bigg[Z_{n-1}\le \frac{-c\sqrt{d}+o(\sqrt{d})}{\sqrt{2d+c\sqrt{d}+o(\sqrt{d})}}\bigg]\\
&	=\bP\bigg[Z_{n-1}\le -\frac{c}{\sqrt{2}}+o(1)\bigg]\ton \Phi\Big(-\frac{c}{\sqrt{2}}\Big)
\end{align*}
and in the same way also
\begin{align*}
\bP[\Bin(n-k-1,1/2)\le d-k-1]
&	=\bP\bigg[Z_{n-k-1}\le \frac{2(d-k-1)-(n-k-1)}{\sqrt{n-k-1}}\bigg]\\
&	=\bP\bigg[Z_{n-k-1}\le -\frac{c+b}{\sqrt{2}}+o(1)\bigg]\ton \Phi\Big(-\frac{c+b}{\sqrt{2}}\Big).
\end{align*}
Combining this with~\eqref{eq:k-faces_covefron_as_binomial_probs1} completes the proof.
\end{proof}

\subsection{Large deviation principles}\label{section:LDP_f_k}

In this section, we state a kind of large deviation principle for the expected face numbers of both $D_{n,d}$ and $C_{n,d}$, as $n$, $d$ and $k$ tend to infinity in a linearly coordinated way. That is, we consider the regime where $k=k(d)$ and $n=n(d)$ are such that
\begin{align}\label{eq:AssumptionGrowth}
{k\over d}\to\varrho\quad\text{and}\quad{d\over n}\to\delta,\qquad\text{as }d\to\infty,
\end{align}
for parameters $\varrho,\delta\in(0,1)$.
%\CT[Was ist mit den Randfaellen $0$ und $1$?]
We start with the result for the expected number  of $k$-faces of $D_{n,d}$. It is illustrated in Figure~\ref{pic:LDP_f_k_Donoho} and uncovers a threshold phenomenon around the point where $\varrho=2-1/\delta$.

\begin{theorem}\label{thm:DonohoTannerLDPfk}
Consider the Donoho-Tanner random cone $D_{n,d}$ and suppose that \eqref{eq:AssumptionGrowth} holds. Then
\begin{align*}
&\lim_{d\to\infty}{1\over d}\log\bE f_k(D_{n,d}) = \begin{cases}
-{1\over\delta}\log(2-2\delta)+\varrho\log\big({2\over\varrho}-2\big)+\log\big({1-\delta\over\delta(1-\varrho)}\big) &: \varrho>2-1/\delta,\\
%({1\over\delta}-\varrho)\log 2+({1\over\delta}-1)\log({1\over\delta}-1)+(1-\varrho)\log(1-\varrho)-{1\over\delta}\log{1\over\delta}+\varrho\log\varrho &: \varrho>2-1/\delta\\
{1\over\delta}\log{1\over\delta}-\varrho\log\varrho-({1\over\delta}-\varrho)\log({1\over\delta}-\varrho) &: \varrho<2-1/\delta.
\end{cases}
\end{align*}
\end{theorem}

\begin{figure}[t]
\centering
\includegraphics[scale=0.5]{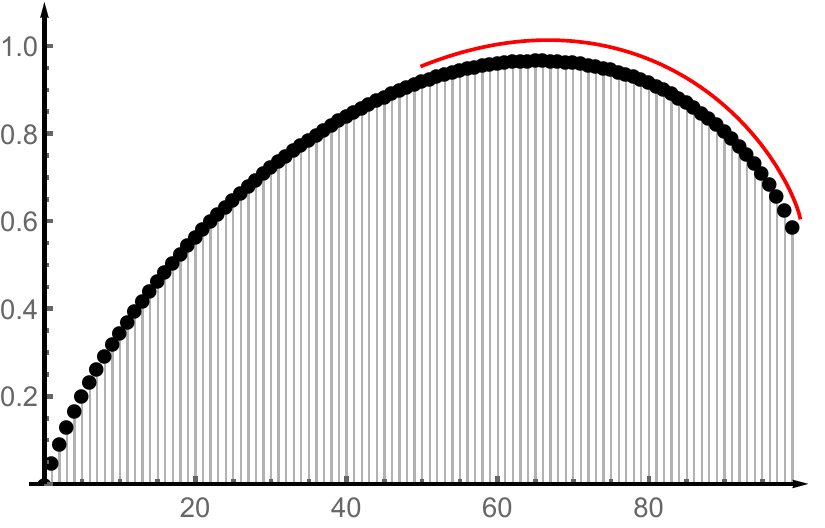}
\quad
\includegraphics[scale=0.5]{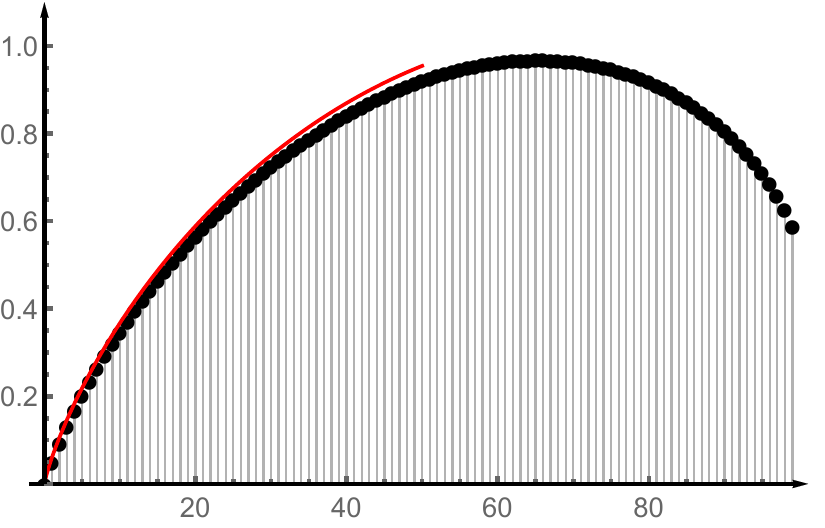}
\caption{\textbf{In black}: Values of ${1\over d}\log \bE f_k(D_{n,d})$ for $\delta=2/3$ and $d=100$; \textbf{In red}: Information functions in the corresponding case. Left panel: Information function for $\varrho>2-1/\delta=1/2$ (that is, $k>d/2$);
Right panel: Information function for $\varrho<2-1/\delta = 1/2$ (that is, $k<d/2$).
}\label{pic:LDP_f_k_Donoho}
\end{figure}

\begin{proof}
Recalling that $\Bin(m,p)$ stands for a binomial random variable with parameters $m\in\NN$ and $p\in(0,1)$, the expression in \eqref{eq:DonohoTannerF_k} can be rewritten in the form
\begin{align*}
{\bE f_k(D_{n,d})\over{n\choose k}} &= 1-\bP[\Bin(n-k-1,1/2)<n-d] = \bP[\Bin(n-k-1,1/2)\geq n-d].
\end{align*}
Now, put
$$
a := \frac{1/\delta-1}{1/\delta-\varrho}
$$
and observe that, as $d\to\infty$, we have $n-d=(n-k-1)(a+o(1))$. Assuming first that $a>1/2$, which is equivalent to $\varrho>2-1/\delta$, we may now apply \eqref{eq:Cramer>1/2} with $m=n-k-1$ to conclude that
\begin{align*}
&\lim_{d\to\infty}{1\over d}\log \bP[\Bin(n-k-1,1/2)\geq n-d]\\
&= \lim_{d\to\infty}{1\over d}\log\bP[\Bin[n-k-1,1/2]\geq (n-k-1)(a+o(1))]\\
&= \lim_{d\to\infty}{n-k-1\over d}{1\over n-k-1}\log\bP[\Bin[n-k-1,1/2]\geq (n-k-1)(a+o(1))]\\
&=-\Big({1\over\delta}-\varrho\Big)\cI(a),
\end{align*}
since $(n-k-1)/d$ tends to $(1/\delta)-\varrho$, as $d\to\infty$, by assumption \eqref{eq:AssumptionGrowth}.
%Also note that $\cI(a)$ is a function of the two parameters $\delta$ and $\rho$ and will be denoted by $\cJ_1(\delta,\varrho)$ in what follows.
Note that $\cI(a)$ is given by
\begin{align*}
%\cJ_1(\delta,\varrho)
\cI(a)= \cI\bigg({{1\over\delta} - 1\over {1\over\delta} - \varrho}\bigg) = -{1\over 1-\delta\varrho}\Big[\log\Big({1-\delta\varrho\over 2-2\delta}\Big)+\delta\log\Big({1-\delta\over\delta(1-\varrho)}\Big)+\delta\varrho\log\Big({2\delta(1-\varrho)\over 1-\delta\varrho}\Big)\Big],
\end{align*}
due to~\eqref{eq:Cramer_information}. To conclude the result in this case, it remains to observe that
\begin{align*}
\log \bE f_k(D_{n,d}) = \log{\bE f_k(D_{n,d})\over{n\choose k}} + \log{n\choose k}.
\end{align*}
The second summand on the right-hand side, divided by $d$, converges to the follwing function depending on $\delta$ and $\varrho$:
\begin{equation}\label{eq:LDPBinomialCoeff}
\lim_{d\to\infty}{1\over d}\log{n\choose k} ={1\over\delta}\log{1\over\delta}-\varrho\log\varrho-\Big({1\over\delta}-\varrho\Big)\log\Big({1\over\delta}-\varrho\Big) =:\cJ(\delta,\varrho).
\end{equation}
Indeed, using three times Stirling's formula together with our assumption on the linear growth of $k$ and $n$ relative to $d$ we see that
\begin{align*}
\lim_{d\to\infty}{1\over d}\log{n\choose k} &= \lim_{d\to\infty}{1\over d}\log{(n/e)^n\over(k/e)^k((n-k)/e)^{n-k}}\\
&=\lim_{d\to\infty}{1\over d} \Big[{d\over\delta}\log{d\over\delta}-\varrho d\log(\varrho d)-d\Big({1\over\delta}-\varrho\Big)\log\Big(d\Big({1\over\delta}-\varrho\Big)\Big)\Big]\\
&={1\over\delta}\log{1\over\delta}-\varrho\log\varrho-\Big({1\over\delta}-\varrho\Big)\log\Big({1\over\delta}-\varrho\Big).
\end{align*}
Thus, using~\eqref{eq:LDPBinomialCoeff} we conclude that
\begin{align*}
\lim_{d\to\infty}{1\over d}\log\bE f_k(D_{n,d}) &= \lim_{d\to\infty}{1\over d}\log{\bE f_k(D_{n,d})\over{n\choose k}}+\lim_{d\to\infty}{1\over d}\log{n\choose k}\\
&= -\Big({1\over\delta}-\varrho\Big)\cI(a) + \cJ(\delta,\varrho).
\end{align*}
Simplifying this expression yields the result.

On the other hand, if $a<1/2$, which is equivalent to $\varrho<2-1/\delta$, we have that
\begin{align*}
\lim_{d\to\infty}{1\over d}\log \bE f_k(D_{n,d}) = \lim_{d\to\infty}{1\over d}\log {\bE f_k(D_{n,d})\over{n\choose k}} + \lim_{d\to\infty}{1\over d}\log{n\choose k} = 0 + \cJ(\delta,\varrho).
\end{align*}
This completes the argument.
\end{proof}

Our next result is the analogue to Theorem~\ref{thm:DonohoTannerLDPfk} in the Cover-Efron case with the same assumption on the growth of $n$, $d$ and $k$. It is illustrated in Figure~\ref{pic:LDP_f_kCover}.

\begin{theorem}
Consider the Cover-Efron random cone $C_{n,d}$ and suppose that ~\eqref{eq:AssumptionGrowth} holds. Then
\begin{align*}
\lim_{d\to\infty}{1\over d}\log \bE f_k(C_{n,d}) = \begin{cases}
{1\over\delta}\log{1\over\delta}-\varrho\log\varrho-\big({1\over\delta}-\varrho\big)\log\big({1\over\delta}-\varrho\big) &: \delta>{1\over 2}\text{ and }\varrho<2-{1\over\delta},\\
-{1\over\delta}\log(2-2\delta)+\varrho\log\big({2\over\varrho}-2\big)+\log\big({1-\delta\over\delta(1-\varrho)}\big) &: \delta>{1\over 2}\text{ and }\varrho>2-{1\over\delta},\\
\varrho\log\big({2\over\varrho}-2\big)-\log(1-\varrho) &:\delta<{1\over 2}.
\end{cases}
\end{align*}
\end{theorem}

\begin{figure}[t]
\centering
\includegraphics[scale=0.39]{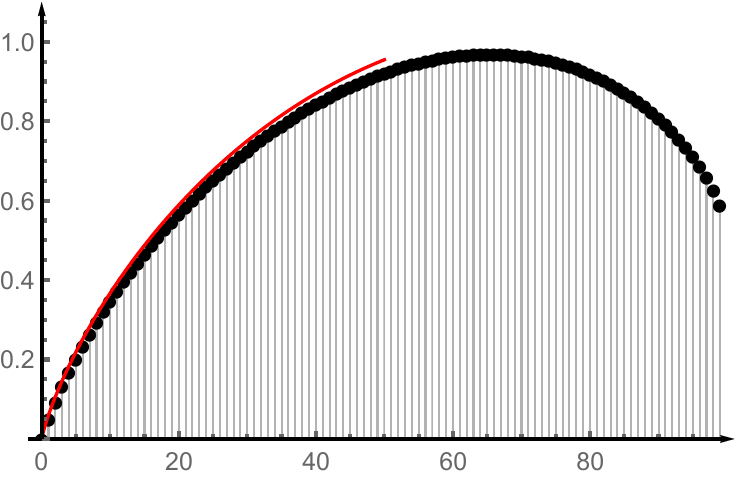}
\quad
\includegraphics[scale=0.39]{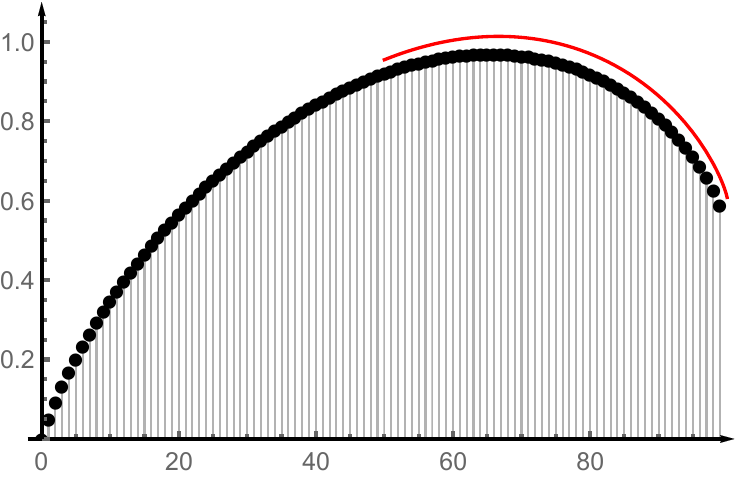}
\quad
\includegraphics[scale=0.39]{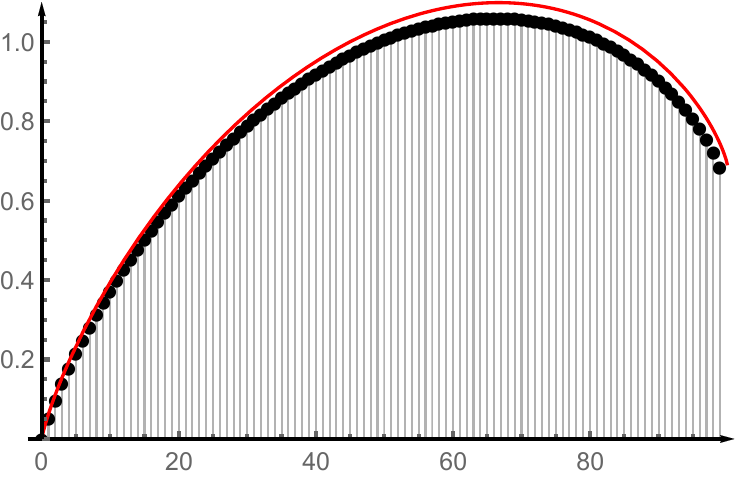}
\caption{\textbf{In black}: Values of ${1\over d}\log \bE f_k(C_{n,d})$ for $d=100$;  {\textbf{In red}}: Information functions in the corresponding case. Left panel: $\delta=2/3$ and information function for $\varrho<2-1/\delta=1/2$
(that is, $k<d/2$);
Middle panel: $\delta=2/3$ and information function for $\varrho>2-1/\delta = 1/2$ (that is, $k>d/2$);
Right panel: $\delta= 1/3$.
}\label{pic:LDP_f_kCover}
\end{figure}

\begin{proof}
%As in the proof of Theorem \ref{thm:DonohoTannerLDPfk} the starting point is a representation of $\bE f_k(C_{n,n})$ in terms of binomial probabilities. In fact, from \eqref{eq:CoverEfronF_k} together with the definition of binomial random variables it follows that
%$$
%{\bE f_k(C_{n,n})\over{n\choose k}} = {\bP[\Bin(n-k-1,1/2)\geq n-n]\over\bP[\Bin(n-1,1/2)\leq n-1]}.
%$$
%\TG[Ich glaube der Zähler ist falsch ($n-n$ durch $n-k-1$ ersetzen)]
Like in the proof of Theorem~\ref{theorem:limit_theorem_faces_cover-efron}, we put
$$
a:=\frac{1-\varrho}{1/\delta-\varrho}
$$
and obtain $d-k-1=(a+o(1))(n-k-1)$.
Using the representation~\eqref{eq:k-faces_covefron_as_binomial_probs} of $\bE f_k(C_{n,d})$ as binomial coefficients yields
\begin{align*}
{1\over d}\log {\bE f_k(C_{n,d})\over{n\choose k}}
&	= {1\over d}\log\bP[\Bin(n-k-1,1/2)\le d-k-1]-{1\over d}\log\bP[\Bin(n-1,1/2)\leq d-1]\\
&={n-k-1\over d}{1\over n-k-1}\log\bP[\Bin(n-k-1,1/2)\leq (n-k-1)(a+o(1))]\\
&\qquad-{n-1\over d}{1\over n-1}\log\bP[\Bin(n-1,1/2)\leq (n-1)(\delta+o(1))].
\end{align*}
%\begin{align*}
%&\lim_{d\to\infty}{1\over d}\log {\bE f_k(C_{n,d})\over{n\choose k}} \\
%&\quad= \lim_{d\to\infty}{1\over d}\log\bP[\Bin(n-k-1,1/2)\le d-k-1]\\
%&\quad\qquad\qquad-\lim_{d\to\infty}{1\over d}\log\bP[\Bin(n-1,1/2)\leq d-1]\\
%&\quad=\lim_{d\to\infty}{n-k-1\over d}{1\over n-k-1}\log\bP[\Bin(n-k-1,1/2)\leq (n-k-1)(a+o(1))]\\
%&\quad\qquad\qquad-\lim_{d\to\infty}{n-1\over d}{1\over n-1}\log\bP[\Bin(n-1,1/2)\leq (n-1)(\delta+o(1))].
%\end{align*}
Since $(n-k-1)/d\to(1/\delta)-\varrho$ and $(n-1)/d\to 1/\delta$, as $d\to\infty$, we conclude from~\eqref{eq:Cramer<1/2} that
%\TG[Ab hier falsch! Auch auf die Definition von $\cJ_1(\delta,\varrho)$ archten!!]
\begin{align*}
\lim_{d\to\infty}{1\over d}\log {\bE f_k(C_{n,d})\over{n\choose k}} = -\Big({1\over\delta}-\varrho\Big)\cI(a) + {1\over\delta}\cI(\delta)
%	= -\Big({1\over\delta}-\varrho\Big)\cJ_1(\delta,\varrho) + {1\over\delta}\cI(\delta),
\end{align*}
provided that $a<1/2$ and $\delta<1/2$. Since $a<1/2$ is equivalent to $\varrho>2-1/\delta$, the latter condition is automatically fulfilled as long as $\delta<1/2$. Note that the function $\cI(a)$ is defined in~\eqref{eq:Cramer_information}. Now, assume that still $\delta>1/2$, but $a<1/2$. Then
\begin{align*}
\lim_{d\to\infty}{1\over d}\log {\bE f_k(C_{n,d})\over{n\choose k}} = -\Big({1\over\delta}-\varrho\Big)\cI(a).
% + 0 = -\Big({1\over\delta}-\varrho\Big)\cJ_1(\delta,\varrho).
\end{align*}
On the other hand, if $a>1/2$ and $\delta>1/2$, then the limit is just zero. Combining these observations with \eqref{eq:LDPBinomialCoeff}, we obtain that
\begin{align*}
\lim_{d\to\infty}{1\over d}\log {\bE f_k(C_{n,d})} &= \lim_{d\to\infty}{1\over d}\log {\bE f_k(C_{n,d})\over{n\choose k}} + \lim_{d\to\infty}{1\over d}\log{n\choose k}\\
&=\begin{cases}
\cJ(\delta,\varrho) &: \delta>{1\over 2}\text{ and }\varrho<2-{1\over\delta},\\
-\big({1\over\delta}-\varrho\big)\cI(a) + \cJ(\delta,\varrho) &: \delta>{1\over 2}\text{ and }\varrho>2-{1\over\delta},\\
-\big({1\over\delta}-\varrho\big)\cI(a) + {1\over\delta}\cI(\delta) + \cJ(\delta,\varrho) &: \delta<{1\over 2},\\
\end{cases}
\end{align*}
where $\cJ(\delta,\varrho)$ is defined in~\eqref{eq:LDPBinomialCoeff}.
Simplification of the right hand side completes the argument.
\end{proof}

\subsection{Application to random Gale diagrams}

A new model to generate combinatorially isomorphic polytopes by choosing their so-called Gale diagram at random has recently been introduced by Schneider \cite{SchneiderGale}, taking up a suggestion of Gale \cite{Gale}. For completeness we recall that with $n\geq d+1$ points $x_1,\ldots,x_n\in\RR^d$ with affine hull equal to $\RR^d$ one can associate points $\bar x_1,\ldots,\bar x_n\in\RR^{n-d-1}$ which linearly span $\RR^{n-d-1}$, a so-called \textit{Gale transform} of $x_1,\ldots,x_n$, as follows. Consider the $(d+1)\times n$-matrix
$$
A := \begin{pmatrix}
	1 & 1 & \cdots & 1 \\
	| & | &  & |\\
	x_1 & x_2 & \cdots & x_n\\
	| & | &  & |
\end{pmatrix}.
$$
By assumption on $x_1,\ldots,x_n$ the kernel $\ker(A)$ of $A$ has dimension $n-d-1$. Let $\tilde x_1,\ldots,\tilde x_{n-d-1}$ be a basis of $\ker(A)$. If $\bar A$ denotes the matrix with columns $\tilde x_1,\ldots,\tilde x_{n-d-1}$ then $A\bar A=0$. The $n$ ordered rows $\bar x_1,\ldots,\bar x_n$ of $\bar A$ are called a \textit{Gale transform} of $x_1,\ldots,x_n$, and by a \textit{Gale diagram} one understands the vector configuration $\{\bar x_1,\ldots,\bar x_n\}$ drawn in $\RR^{n-d-1}$. The vectors $\bar x_1,\ldots,\bar x_n$ satisfy $\bar x_1+\ldots+\bar x_n=o$ with $o$ being the origin of $\RR^{n-d-1}$. Of course, Gale transforms are not unique, since the choice of a basis for $\ker(A)$ was arbitrary. On the other hand, all choices for $\bar A$ differ by multiplication with a non-singular matrix. Since we are interested in the combinatorial quantities only, the precise choice for $\bar A$ is therefore irrelevant for us. We also remark that a collection $\bar x_1,\ldots,\bar x_n$ of vectors in $\RR^{n-d-1}$ with $\pos(\bar x_1,\ldots,\bar x_n)=\RR^{n-d-1}$ is the Gale diagram of a sequence $x_1,\ldots,x_n\in\RR^d$ (more precisely, of many sequences, but their convex hulls are combinatorially all equivalent). We refer, for example, to the monographs \cite{GrunbaumBook,McMullenBook} for further background material on Gale transforms and diagrams.

To recall the model suggested in \cite{SchneiderGale},
%as in Section \ref{section:random_cones}
we let $X_1,\ldots,X_n$ be $n\geq d+1$ independent $(n-d-1)$-dimensional random vectors distributed according to a probability measure $\phi$ in $\RR^{n-d-1}$  which has the property that $\phi$ is even and puts mass zero to each linear hyperplane. Following \cite{SchneiderGale}, a $(\phi,n)$ \textit{random Gale diagram} is the collection of random points $X_1,\ldots,X_n$, conditionally on the event that the origin $o$ belongs to the (interior) of their convex hull. As a consequence, there are almost surely random scalars $\lambda_1,\ldots,\lambda_n>0$ such that $\bar a_1:=\lambda_1X_1,\ldots,\bar a_n:=\lambda_nX_n$ satisfy $\bar a_1+\ldots+\bar a_n=o$. Therefore, $\bar a_1,\ldots,\bar a_n$ is the Gale transform of some random points $a_1,\ldots,a_n$ in $\RR^d$ whose convex hull is denoted by $G_{n,d}$. Following~\cite{SchneiderGale}, $G_{n,d}$ is called a \textit{random Gale polytope}. As explained above, this does not define the random Gale polytope $G_{n,d}$ uniquely. Rather is defines a whole class of random polytopes. On the other hand, all random polytopes arising in this way are combinatorially equivalent, which implies that for any $k\in\{0,1,\ldots,d-1\}$ the random variable $f_k(G_{n,d})$ is well defined. Since we are interested in these quantities only, without ambiguity we refer to $G_{n,d}$ as \textit{the} random Gale polytope in $\RR^d$ generated by $n$ points.

Theorem 2 in \cite{SchneiderGale} uncovers a threshold phenomenon for the expected number of $k$-dimensional faces of the random Gale polytopes $G_{n,d}$ in high dimensions, more precisely, if $d$, $n$ and $k$ tend to infinity in a linearly coordinated way (this is rephrased as \eqref{eq:first_part_schneider} below). As an application of the results developed in this section we can strengthen and refine this as follows.

\begin{corollary}
Let $G_{n,d}$ be a random Gale polytope. Suppose that $k=k(d)$ and $n=n(d)$ are such that
$$
{d\over n} \to\delta\qquad\text{and}\qquad{k\over d}\to\varrho,\qquad{\rm as}\;\;d\to\infty
$$
for $\delta\in(0,1)$ and $\varrho\in(0,1)$. Then, it holds that
\begin{align}\label{eq:first_part_schneider}
	\lim_{d\to\infty}\frac{\bE f_k(G_{n,d})}{{n\choose k+1}}=
	\begin{cases}
		1		&: \delta>1/2\text{ and }\varrho<2-1/\delta,\\
		0		&: \delta>1/2\text{ and }\varrho>2-1/\delta,\\
		0		&: \delta\le 1/2.
	\end{cases}
\end{align}
Furthermore, if $\delta>1/2$ and
\begin{align*}
	n=\frac d\delta+c\sqrt d +o(\sqrt d)\quad\text{and}\quad k=\Big(2-\frac 1\delta\Big)d+b\sqrt{d}+o(\sqrt{d}),\qquad\text{as }d\to\infty,
\end{align*}
for parameters $c,b\in\RR$, then it holds that
\begin{align*}
	\lim_{d\to\infty}\frac{\bE f_k(G_{n,d})}{{n\choose k+1}}=\Phi\bigg(-\frac{c+b}{\sqrt{2/\delta-2}}\bigg).
\end{align*}
In the case where
\begin{align*}
	n=2d+c\sqrt{d}+o(\sqrt{d})\quad\text{and}\quad k=b\sqrt{d}+o(\sqrt{d}),\qquad\text{as }d\to\infty,
\end{align*}
for constants $c\in\RR$ and $b\geq 0$, it holds that
\begin{align*}
	\lim_{d\to\infty}\frac{\bE f_k(G_{n,d})}{{n\choose k+1}}=\frac{\Phi(-(c+b)/\sqrt{2})}{\Phi(-c/\sqrt{2})}.
\end{align*}
\end{corollary}
\begin{proof}
Using the description of faces by means of Gale diagrams, it has been shown in \cite[Equation (3)]{SchneiderGale} that $\bE f_k(G_{n,d})$ can be expressed in terms of binomial probabilities:
$$
{\bE f_k(G_{n,d})\over{n\choose k+1}} ={\bP[\Bin(n-k-2,1/2)\leq d-k-1]\over\bP[\Bin(n-1,1/2)\leq d]}.
$$
Comparing this to \eqref{eq:k-faces_covefron_as_binomial_probs} yields that
$$
\bE f_k(G_{n,d})=\bE f_{k+1}(C_{n,d+1}),
$$
where $C_{n,d+1}$ is the Cover-Efron random cone in $\RR^{d+1}$ generated by $n$ points, see also \cite[Equation (4)]{SchneiderGale}. The result thus follows from that on the Cover-Efron random cone in Theorem \ref{theorem:limit_theorem_faces_cover-efron}.
\end{proof}

\section{Limit theorems for the expected conic intrinsic volumes and quermassintegrals}

This section covers various limit theorems  for the expected conic intrinsic volumes  and conic quermassintegrals of the random cones introduced in Section~\ref{section:random_cones}.

\subsection{Distributional limit theorems for the conic intrinsic volumes}

Here, we state distributional limit theorems, including central limit theorems, for random variables that take value $k\in\{0,1,\ldots,d\}$ with probability $\bE\upsilon_k(D_{n,d})$ in the Donoho-Tanner case, and $\bE\upsilon_k(C_{n,d})$ in the Cover-Efron case.

Starting with the Donoho-Tanner random cone $D_{n,d}$, we first provide an explicit formula for $\bE\upsilon_k(D_{n,d})$, which we couldn't locate in the existing literature.

%%%%%%%% falscher Beweis(/Lemma) befindet sich hinter \end{document}
\begin{lemma}\label{lem:ConicIntVolDonohoTanner}
Fix integers $0<d\leq n$. Let $D_{n,d}$ be a Donoho-Tanner random cone. Then, it holds that
\begin{align*}
\bE v_k(D_{n,d}) = \begin{cases}
{1\over 2^n}{n\choose k} &: k\in\{0,\ldots,d-1\},\\
1-\sum_{j=0}^{d-1} \frac{1}{2^n}\binom nj&: k=d.
\end{cases}
\end{align*}
\end{lemma}

\begin{proof}
Recall from Section~\ref{section:random_cones} that the Donoho-Tanner random cone $D_{n,d}$, conditioned on the event that $D_{n,d}\neq\RR^d$, has the same distribution as the Cover-Efron random cone $C_{n,d}$. Thus, we can retrace the formulas for the expected conic intrinsic volumes of $D_{n,d}$ back to~\eqref{eq:intr_vol_cover-efron}. For $k\in\{0,\dots,d-1\}$, this yields
\begin{align*}
\bE[\upsilon_k(D_{n,d})]
&	=\bE\big[\upsilon_k(D_{n,d})\ind_{\{D_{n,d}\neq\RR^d\}}\big]+\bE\big[\upsilon_k(\RR^d)\ind_{\{D_{n,d}=\RR^d\}}\big]\\
&	=\bP[D_{n,d}\neq\RR^d]\cdot\bE[\upsilon_k(D_{n,d})|D_{n,d}\neq\RR^d].
\end{align*}
%where $\delta_{k,d}$ denotes the Kronecker delta.
Now, we can use~\eqref{eq:intr_vol_cover-efron} and Wendel's formula $\bP[D_{n,d}\neq \RR^d]= C(n,d)/2^n$, see~\cite[Theorem~8.2.1]{SW} or~\cite{Wendel}, to obtain
\begin{align*}
\bE[\upsilon_k(D_{n,d})]=\frac{C(n,d)}{2^n}\cdot\binom{n}{k}\frac{1}{C(n,d)}=\binom nk\frac 1{2^n},
\end{align*}
for $k\in\{0,\dots,d-1\}$.
Using that $\upsilon_0(C)+\ldots+\upsilon_d(C)=1$ for a cone $C\subset\RR^d$, we obtain
\begin{align*}
\bE[\upsilon_d(D_{n,d})]=1-\sum_{j=0}^{d-1}\binom nj \frac{1}{2^n},
\end{align*}
which completes the proof.
\end{proof}

We are now prepared to present our main result for the expected conic intrinsic volumes of the Donoho-Tanner random cones, which is illustrated in Figure~\ref{pic:v_k_Donoho}.

\begin{theorem}\label{theorem:CLT_v_k_Donoho-Tanner}
Let $D_{n,d}$ be a Donoho-Tanner random cone and let $X_{n,d}$ be a random variable with probability law $\bP[X_{n,d}=k]=\bE \upsilon_k(D_{n,d})$ for $k=0,\dots,d$.
Suppose $n=n(d)$ is such that
\begin{align*}
\frac dn\longrightarrow \delta,\qquad\text{as }d\to\infty,
\end{align*}
for a parameter $\delta\in (1/2,1)$. Then, we obtain the central limit theorem
\begin{align*}
\frac{2X_{n,d}-n}{\sqrt{n}}\todistr N(0,1).
\end{align*}
In the case where $\delta=1/2$, and more precisely, where
$$
n=2d+c\sqrt{d}+o(\sqrt{d}),\qquad\text{as }d\to\infty,
$$
for a parameter $c\in\RR$, it holds that
\begin{align*}
\frac{2X_{n,d}-n}{\sqrt{n}}\todistr Z,
\end{align*}
where $Z$ is a random variable with distribution
\begin{align*}
\bP[Z\le t]=\Phi(t),\quad t<-\frac c{\sqrt 2},\qquad\text{and}\qquad\bP\Big[Z=-\frac{c}{\sqrt 2}\Big]=\Phi\Big(\frac{c}{\sqrt{2}}\Big).
\end{align*}
\end{theorem}

\begin{figure}[t]
\centering
\includegraphics[scale=0.5]{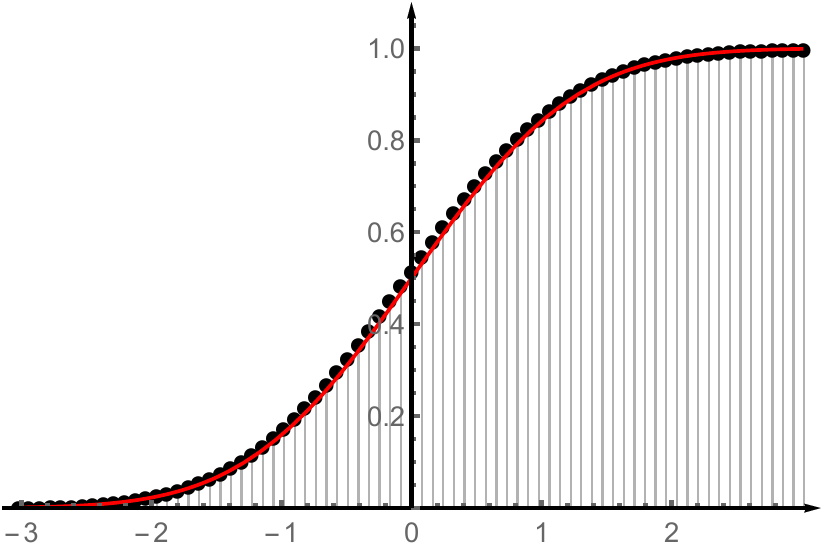}
\quad
\includegraphics[scale=0.5]{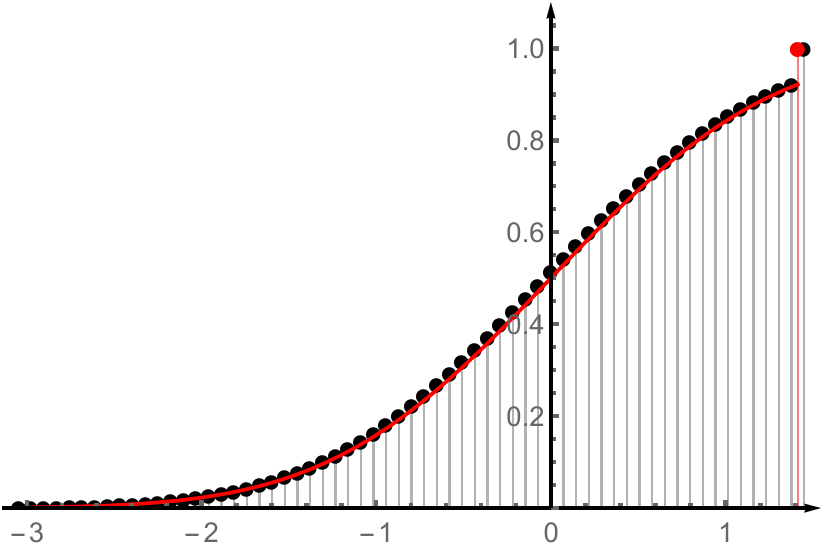}
\caption{\textbf{In black}: Distribution function of $(2X_{n,d}-n)/\sqrt n$ at discontinuity points; {\textbf{In red}}: Limiting distribution function. Left panel: First case with $\delta=2/3$, $d=400$;
Right panel: Second case with $c=-2$, $d=400$.
}\label{pic:v_k_Donoho}
\end{figure}

\begin{proof}
For $\delta\in (1/2,1)$ and $t\in\RR$, we have
\begin{align*}
\bP\bigg[\frac{2X_{n,d}-n}{\sqrt{n}}\le t\bigg]
	=\bP\left[X_{n,d}\le \frac{n}{2}+t\sqrt{\frac{n}{4}}\right]
%	=\bP\left[X_{n,d}\le \frac{d}{2\delta}+t\sqrt{\frac{d}{4\delta}}\,\right]
%&	=\bP\left[\Bin(n,1/2)\le \frac{d}{2\delta}+t\sqrt{\frac{d}{4\delta}}\,\right] \\
	=\bP\left[\Bin(n,1/2)\le \frac{n}{2}+t\sqrt{\frac{n}{4}}\right],
\end{align*}
where we used that $\frac{n}{2}+t\sqrt{\frac{n}{4}}<d$ for sufficiently large $d$ and in this case, the $k$-th conic intrinsic volume of $D_{n,d}$ coincides with the probability of the event $\{\Bin(n,1/2)=k\}$, see Lemma~\ref{lem:ConicIntVolDonohoTanner}. Then, the central limit theorem for the binomial distribution yields the claim.

Now, let $n=2d+c\sqrt{d}+o(\sqrt{d})$ for a parameter $c\in\RR$. In the case $t<-c/\sqrt{2}$, we observe that for $d$ sufficiently large, we have
\begin{align*}
\frac{n}{2}+t\sqrt{\frac n4}
	=d+\frac c2\sqrt{d}+t\sqrt{\frac d2+\frac c4\sqrt d+o(\sqrt{d})}+o(\sqrt d)
	=d+\bigg(\frac c2+\frac t{\sqrt{2}}\bigg)\sqrt{d}+o(\sqrt d)
	<d.
\end{align*}
Thus, we obtain
\begin{align*}
\bP\left[X_{n,d}\le \frac{n}{2}+t\sqrt{\frac{n}{4}}\right]=\bP\left[\Bin(n,1/2)\le \frac{n}{2}+t\sqrt{\frac{n}{4}}\right]\to\Phi(t)
\end{align*}
for $t<-c/\sqrt 2$. In the case $t> -c/\sqrt{2}$, we have that $\frac{n}{2}+t\sqrt{\frac n4}> d$ for $d$ sufficiently large, and the probability converges to $1$, which yields the claim.
\end{proof}

Similarly to the Donoho-Tanner case, we can also derive distributional limit theorems for the expected intrinsic volumes of the Cover-Efron random cone. This is illustrated in Figure~\ref{pic:v_k}.

\begin{figure}[!ht]
\centering
\includegraphics[scale=0.4]{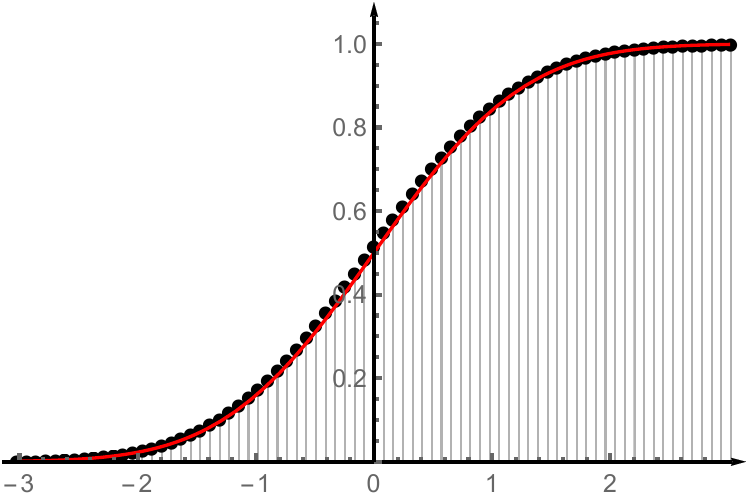}
\quad
\includegraphics[scale=0.4]{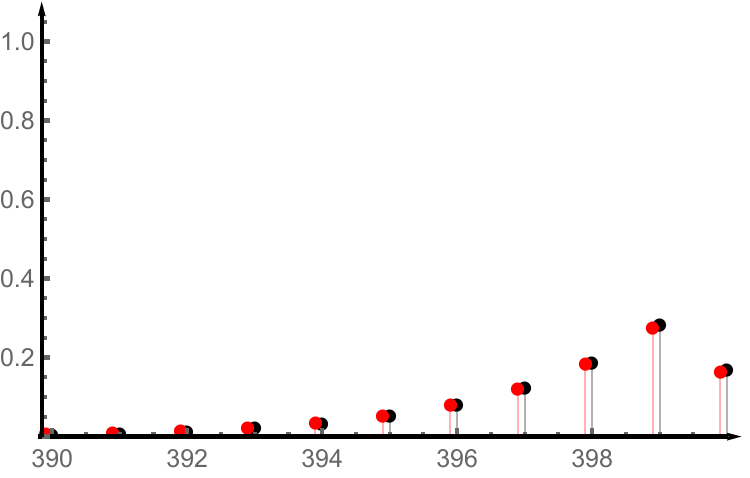}
\includegraphics[scale=0.4]{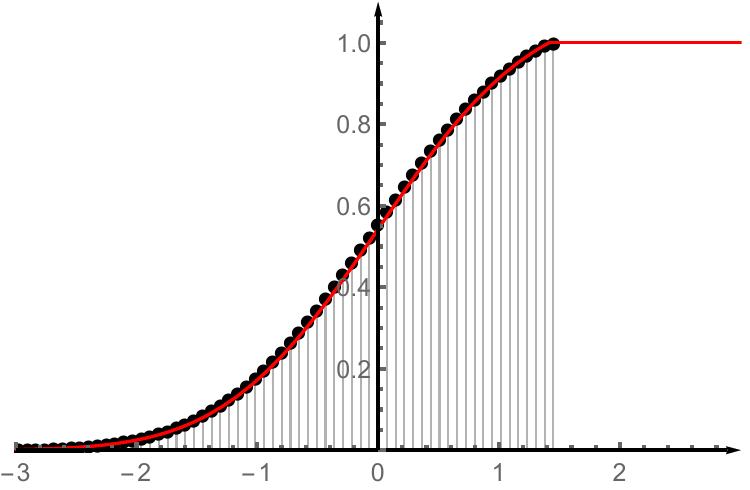}
\caption{Illustration of Theorem \ref{theorem:v_k_CLT_Cover-Efron} for $d=400$. Left panel: The case $\delta=2/3$; distribution function of $2(X_{n,d}-n)/\sqrt n$ in \textbf{black}, limiting distribution function in {\textbf{red}}.
Middle panel: The case $\delta=2/5$, counting density of $X_{n,d}$ in \textbf{black}, counting density of $d-Z$ in {\textbf{red}}.
Right panel: The critical case $\delta=1/2$, $c=-2$; distribution function of $2(X_{n,d}-n)/\sqrt n$ in \textbf{black}, limiting distribution function in {\textbf{red}}.
}\label{pic:v_k}
\end{figure}

\begin{theorem}\label{theorem:v_k_CLT_Cover-Efron}
Let $C_{n,d}$ be a Cover-Efron random cone and let $X_{n,d}$ be the random variables with probability law $\bP[X_{n,d}=k]=\bE \upsilon_k(C_{n,d})$ for $k=0,\dots,d$. Suppose that $n=n(d)$ is such that
\begin{align*}
\frac dn\longrightarrow \delta,\qquad\text{as }d\to\infty,
\end{align*}
for a parameter $\delta\in (0,1)$. In the case where $\delta\in (1/2,1)$, we obtain the central limit theorem
\begin{align*}
\frac{2X_{n,d}-n}{\sqrt{n}}\todistr N(0,1).
\end{align*}
In the case where $\delta\in(0,1/2)$, it holds that
\begin{align*}
d-X_{n,d}\todistr Z,
\end{align*}
where $Z$ is a random variable with values in $\NN_0$ whose distribution is given by
\begin{align}\label{eq:Z_distr_almost_geo}
\bP[Z=0]=\frac 12 \cdot\frac{1-2\delta}{1-\delta},\qquad\text{and}\qquad\bP[Z=k]=\frac{1}{2\delta}\cdot\frac{1-2\delta}{1-\delta}\cdot\Big(1-\frac{1-2\delta}{1-\delta}\Big)^k,\quad k\in\NN.
\end{align}
For $\delta=1/2$, and more precisely,
$$
n=2d+c\sqrt{d}+o(\sqrt{d}),\qquad\text{as }d\to\infty,
$$
for a parameter $c\in\RR$, we obtain
\begin{align*}
\frac{2X_{n,d}-n}{\sqrt{n}}\todistr N(0,1)\,\Big| \Big\{N(0,1)<-\frac c{\sqrt{2}}\Big\},
\end{align*}
where the latter notation indicates a random variable whose distribution is given by the conditional distribution of $N(0,1)$ given the event $\{N(0,1)<-c/\sqrt{2}\}$.
\end{theorem}

\begin{proof}
For $\delta\in (1/2,1)$, the law of large numbers implies that
\begin{align*}
\lim_{d\to\infty}\bP[\Bin(n-1,1/2)\le d-1]= 1.
\end{align*}
Recall the formulas~\eqref{eq:intr_vol_cover-efron} for the conic intrinsic volumes of $C_{n,d}$. Since
\begin{align*}
\frac n2+t\sqrt{\frac n4}=\frac{d}{2\delta}+o(d)<d
\end{align*}
for all $t\in\RR$ and $d$ sufficiently large, we obtain
\begin{align*}
\bP\left[X_{n,d}\le \frac n2+t\sqrt{\frac n4}\,\right]
&	= \frac 1 {2^n\bP[\Bin(n-1,1/2)\le d-1]}\sum_{k=0}^{\lfloor \frac n2+t\sqrt{\frac n4}\rfloor}\binom nk\\
&	\hspace*{-2.2mm}\tosim \sum_{k=0}^{\lfloor \frac n2+t\sqrt{\frac n4}\rfloor}\frac{\binom nk}{2^n}\\
&	=\bP\left[\Bin(n,1/2)\le \frac n2+t\sqrt{\frac n4}\,\right]\ton \Phi(t)
\end{align*}
from the central limit theorem for binomial random variables. This proves the first claim.

In the case $\delta\in (0,1/2)$ we obtain, for $k\in\NN$,
\begin{align*}
\bP[d-X_{n,d}= k]
&	=\bP[X_{n,d}= d-k]\\
&	=\frac{\bP[\Bin(n,1/2)= d-k]}{\bP[\Bin(n-1,1/2)\le d-1]}\\
&	=\frac{\bP\left[\Bin(n,1/2)= (\frac{d-k}{n})n\right]}{\bP\big[\Bin(n-1,1/2)\le (\frac{d-1}{n-1})(n-1)\big]}.
\end{align*}
Since  both sequences $(d-k)/n$ and $(d-1)/(n-1)$ converge to $\delta\in(0,1/2)$, as $d\to\infty$, we can apply the asymptotic equivalence~\eqref{eq:asym_binomial_>=} in the denominator and~\eqref{eq:asym_binomial_=} in the numerator and arrive at
\begin{align*}
\bP[d-X_{n,d}= k]\tosim \frac{\exp\left\{-n\cdot \cI(\frac{d-k}{n})\right\}}{\exp\big\{-(n-1)\cdot \cI(1-\frac{d-1}{n-1})\big\}}\cdot\frac{2(1-\delta)-1}{1-\delta}.
\end{align*}
Combined with the definition of the information function $\cI(x)$, see~\eqref{eq:Cramer_information}, we obtain
\begin{align*}
\exp\left\{-n\cdot \cI\Big(\frac{d-k}{n}\Big)\right\}
%	=2^{-n}\Big(1-\frac{d-k}{n}\Big)^{d-n-k}\Big(\frac{d-k}{n}\Big)^{k-d}
	=2^{-n}\Big(\frac{n-d+k}{n}\Big)^{d-n-k}\Big(\frac{n}{d-k}\Big)^{d-k}
\end{align*}
and
\begin{align*}
\exp\left\{-(n-1)\cdot \cI\Big(1-\frac{d-1}{n-1}\Big)\right\}
	=2^{-(n-1)}\Big(\frac{n-d}{n-1}\Big)^{d-n}\Big(\frac{n-1}{d-1}\Big)^{d-1}.
\end{align*}
This yields
\begin{align*}
\bP[d-X_{n,d}= k]
&	\tosim \frac {1-2\delta}{2(1-\delta)} \Big(\frac{n-d+k}{n}\Big)^{d-n-k}\Big(\frac{n}{d-k}\Big)^{d-k}\Big(\frac{n-1}{n-d}\Big)^{d-n}\Big(\frac{d-1}{n-1}\Big)^{d-1}\\
&	\hspace*{2.2mm}=\frac {1-2\delta}{2(1-\delta)}\frac{n^n}{(n-1)^{n-1}}\frac{(d-1)^{d-1}}{(d-k)^d}\Big(\frac{n-d}{n-d+k}\Big)^{n-d}\Big(\frac{d-k}{n-d+k}\Big)^k\\
&	\hspace*{2.2mm}=\frac {1-2\delta}{2(1-\delta)} \frac{n-1}{d-k}\Big(1-\frac 1n\Big)^{-n}\Big(1-\frac{k-1}{d-1}\Big)^{1-d}\Big(1+\frac{k}{n-d}\Big)^{d-n}\Big(\frac{d-k}{n-d+k}\Big)^k.
\end{align*}
Now, we use that $d/n\to\delta$ and that
\begin{align*}
\Big(1-\frac 1n\Big)^{-n}\ton e,\quad \Big(1-\frac{k-1}{d-1}\Big)^{1-d}\ton e^{k-1},\quad \Big(1+\frac{k}{n-d}\Big)^{d-n}\ton e^{-k},
\end{align*}
to finally obtain
\begin{align*}
\lim_{d\to\infty}\bP[d-X_{n,d}= k]=\frac 1{2\delta}\cdot\frac{1-2\delta}{1-\delta}\Big(\frac{\delta}{1-\delta}\Big)^k=\frac 1{2\delta}\cdot\frac{1-2\delta}{1-\delta}\Big(1-\frac{1-2\delta}{1-\delta}\Big)^k.
\end{align*}
We consider the case $k=0$ separately. Here, we have
\begin{align*}
\bP[d-X_{n,d}= 0]
	=\bP[X_{n,d}= d]
	=\frac{\binom{n-1}{d-1}}{C(n,d)}
%&	=\frac{\bP[\Bin(n,1/2)\le d-k]}{\bP[\Bin(n-1,1/2)\le d-1]}\\
	=\frac{\bP\big[\Bin(n-1,1/2)= (\frac{d-1}{n-1})(n-1)\big]}{2\bP\big[\Bin(n-1,1/2)\le (\frac{d-1}{n-1})(n-1)\big]}.
\end{align*}
Using~\eqref{eq:asym_binomial_=} and~\eqref{eq:asym_binomial_>=}, we arrive at
\begin{align*}
\lim_{d\to\infty}\bP[d-X_{n,d}= 0]=\frac 12\cdot\frac{2(1-\delta)-1}{1-\delta}=\frac 12\cdot\frac{1-2\delta}{1-\delta},
\end{align*}
which completes the proof of the second claim.

Now, let $n=2d+c\sqrt{d}+o(\sqrt{d})$ for a parameter $c\in\RR$. In the same way as in the proof of Theorem~\ref{theorem:CLT_v_k_Donoho-Tanner} we observe that for  $t<-c/\sqrt{2}$ and $d$ sufficiently large, we have
$
\frac{n}{2}+t\sqrt{\frac n4}
	<d.
$
Furthermore, defining $Z_n:=(2\Bin(n,1/2)-n)/\sqrt n$ and using the central limit theorem for the binomial distribution, we obtain
\begin{align*}
\bP[\Bin(n-1,1/2)\le d-1]
&	=\bP\left[Z_{n-1}\le \frac{2(d-1)-(n-1)}{\sqrt{n-1}}\right]\\
&	=\bP\left[Z_{n-1}\le -\frac{c}{\sqrt{2}}+o(1)\right]\ton\Phi\Big(-\frac{c}{\sqrt{2}}\Big).
\end{align*}
Thus, we obtain
\begin{align*}
\bP\left[X_{n,d}\le \frac{n}{2}+t\sqrt{\frac{n}{4}}\right]
&	= \frac 1 {2^n\bP[\Bin(n-1,1/2)\le d-1]}\sum_{k=0}^{\lfloor \frac n2+t\sqrt{\frac n4}\rfloor}\binom nk\\
%&	\hspace*{-2.2mm}\tosim \sum_{k=0}^{\lfloor \frac n2+t\sqrt{\frac n4}\rfloor}\frac{\binom nk}{2^d}\\
&	\hspace*{-2.2mm}\tosim\frac{1}{\Phi(-c/\sqrt{2})}\bP\left[\Bin(n,1/2)\le \frac n2+t\sqrt{\frac n4}\,\right]\ton \frac{\Phi(t)}{\Phi(-c/\sqrt{2})},
\end{align*}
for $t<-c/\sqrt 2$. In the case $t> -c/\sqrt{2}$, we have that $\frac{n}{2}+t\sqrt{\frac n4}> d$ for $d$ sufficiently large, which yields the claim.
\end{proof}

\subsection{Limit theorems for the conic quermassintegrals}

Recall from~\eqref{eq:conic_Crofton} that the conic intrinsic volumes $\upsilon_k$ and the conic quermassintegrals $U_k$ satisfy the relation
\begin{align*}
2U_k(C)= 2(\upsilon_{k+1}(C)+\upsilon_{k+3}(C)+\ldots)
\end{align*}
for any cone $C$ that is not a linear subspace. Thus, the quermassintegrals $U_k$ are essentially the tail functions of the random variables $X_{n,d}$. Hence, Theorem~\ref{theorem:v_k_CLT_Cover-Efron} suggests that for $k\to\infty$, the expectation $\bE U_k(C)$ behaves like the tail function of the limit distribution given in Theorem~\ref{theorem:v_k_CLT_Cover-Efron}. This is specified by the following theorem and illustrated in Figure~\ref{pic:1}.

\begin{theorem}\label{thm:031220}
Let $C_{n,d}$ be a Cover-Efron random cone. Suppose that $n=n(d)$ and $k=k(d)$ are such that
\begin{align}\label{eq:regime_quermass_delta<1/2}
\frac{d}{n}\to\delta\quad\text{and}\quad k=\frac n2+b\sqrt{\frac n4}+o(\sqrt{n}),\qquad\text{as }d\to\infty,
\end{align}
%\TG[$k$ in Ahängigkeit von $d$ geht nicht oder?? Geht nicht.]
for parameters $\delta\in(1/2,1)$ and $b\in\RR$. Then, it holds that
\begin{align*}
\lim_{d\to\infty}\bE 2U_k(C_{n,d})=1-\Phi(b).
\end{align*}
In the case where
\begin{align}\label{eq:regime_quermass_delta=1/2}
n=2d+c\sqrt d +o(\sqrt d)\quad\text{and}\quad k=\frac n2+b\sqrt{\frac n4}+o(\sqrt{n}),\qquad\text{as }d\to\infty,
\end{align}
for parameters $c,b\in \RR$, it holds that
\begin{align*}
\lim_{d\to\infty}\bE 2U_k(C_{n,d})
=
\begin{cases}
1-\frac{\Phi(b)}{\Phi(-c/\sqrt 2)},
&:
b<- \frac{c}{\sqrt 2},\\
0,
&:b\ge - \frac{c}{\sqrt 2}.
\end{cases}
\end{align*}
%In the case where
%$$
%\frac dn\to \delta,\qquad \text{as }d\to\infty,
%$$
%for a parameter $\delta\in(0,1/2)$, we have that
%\begin{align*}
%\lim_{d\to\infty}\bE 2U_{d-k}(C_{n,d})=\sum_{l=1}^k\bP[Z=l],
%\end{align*}
%for all fixed $k\in\NN$,
%where $Z$ is the random variable from Theorem~\ref{theorem:v_k_CLT_Cover-Efron} with values in $\NN_0$ whose distribution is given by
%\begin{align*}
%\bP[Z=0]=\frac 12 %\cdot\frac{1-2\delta}{1-1\delta},\qquad\bP[Z=k]=\frac{1}{2\delta}\cdot\frac{1-2\delta}{1-\delta}\cdot\Big(1-\frac{1-2\delta}{1-\delta}\Big)^k,\quad %k\in\NN.
%\end{align*}
\end{theorem}

\begin{remark}
In the case $\delta\in (0,1/2)$ which has been omitted above, the conic Crofton formula~\eqref{eq:conic_Crofton} together with Theorem~\ref{theorem:v_k_CLT_Cover-Efron} imply an explicit formula for $\lim_{d\to\infty} \bE 2 U_{d-k}(C_{n,d})$, for fixed $k\in \NN$, in terms of the probabilities given in~\eqref{eq:Z_distr_almost_geo}.
This recovers the result of Theorem 1.4 of~\cite{HugSchneiderThresholdPhenomena}.
\end{remark}

\begin{figure}[t]
\centering
\includegraphics[scale=0.43]{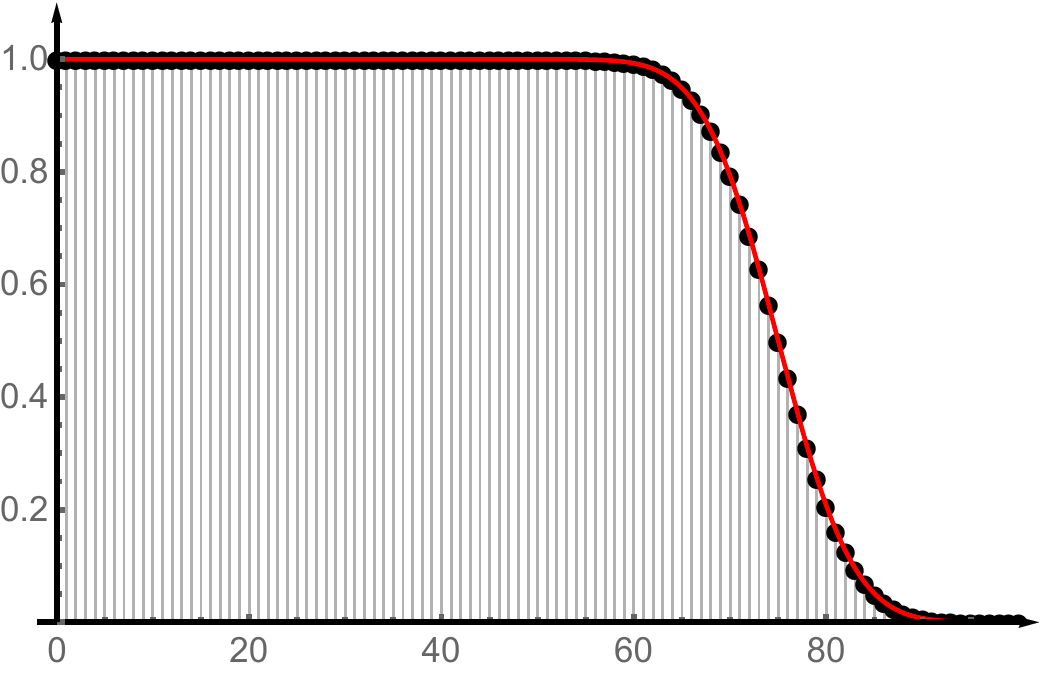}
\quad
\includegraphics[scale=0.43]{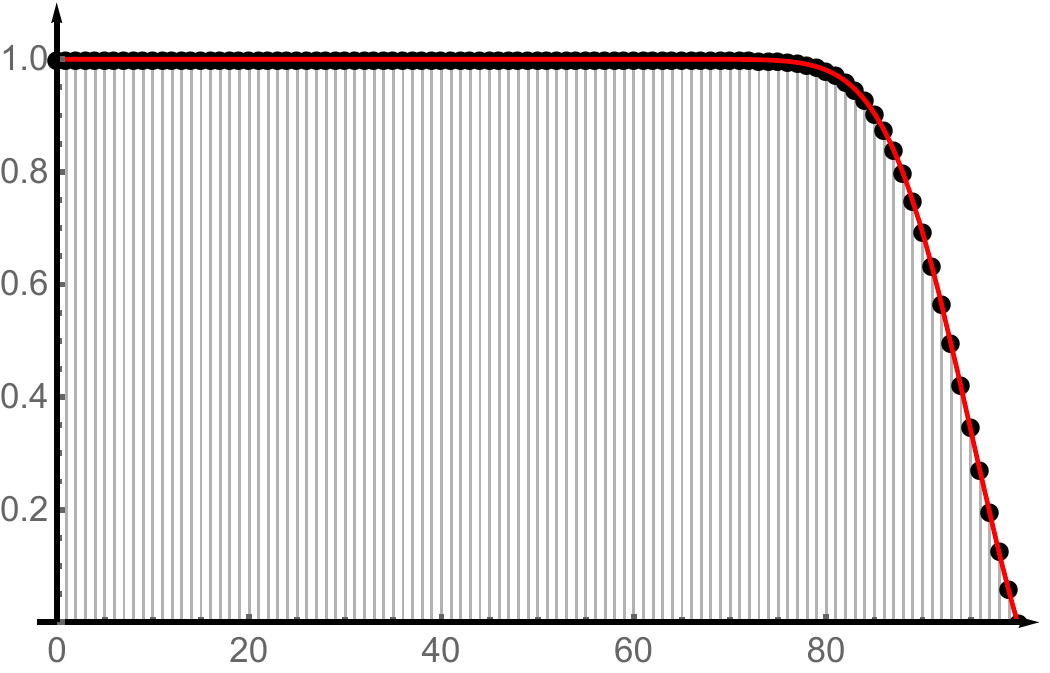}
\caption{Left panel: Convergence in the regime~\eqref{eq:regime_quermass_delta<1/2} with $d=100$ and $\delta=2/3$. Right panel: Convergence in the regime~\eqref{eq:regime_quermass_delta=1/2} with $d=100$ and $c=-1$.
The \textbf{black points} represent the values of $\bE 2U_{k}(C_{n,d})$ for $k=0,1,\dots d$. The {\textbf{curve in red}} is the tail function of the corresponding approximating distribution.
% tail-functional of a standard normal distribution at $(2k-n)/\sqrt{n}$ for $k\in[0,d]$.\newline
% the \textbf{black points} represent the values of $\bE 2U_{k}(C_{n,d})$ for $k=0,1,\dots d$. The {\color{red} \textbf{red line}} is the tail-functional of $N(0,1)|\{N(0,1)<-c/\sqrt 2\}$ at $(2k-n)/\sqrt{n}$ for values $k\in[0,\frac n2 - \frac{c}{\sqrt{2}}\sqrt{\frac n4}]$.
}\label{pic:1}
\end{figure}

\begin{proof}[Proof of Theorem \ref{thm:031220}]
Consider the regime~\eqref{eq:regime_quermass_delta<1/2}. Following~\eqref{eq:quermassint_cover-efron}, we can express the expected quermassintegrals of $C_{n,d}$ in terms of binomial probabilities:
\begin{align}\label{eq:Cover_Efron_quermass_binom_prob}
\bE 2U_{k}(C_{n,d})
	=\frac{C(n,d)-C(n,k)}{C(n,d)}
	=1-\frac{\bP[\Bin(n-1,1/2)\le k-1]}{\bP[\Bin(n-1,1/2)\le d-1]}.
%	=\frac{\sum_{l=1}^k\bP[\Bin(n-1,1/2)=d-l]}{\bP[\Bin(n-1,1/2)\le d-1]}.
\end{align}
%The formula was applicable since
%$$
%k=\frac n2+c\sqrt{\frac n4}+o(\sqrt{n})=\frac{d}{2\delta}+o(d)<d
%$$
%for sufficiently large $d$.
The assumption $\delta>1/2$ yields that the denominator of~\eqref{eq:Cover_Efron_quermass_binom_prob} converges to $1$ as $d\to\infty$, due to the law of large numbers. For the numerator, we can use the central limit theorem for binomial random variables and obtain
\begin{align*}
\bP[\Bin(n-1,1/2)\le k-1]
&	=\bP\left[Z_{n-1}\le \frac{2(k-1)-(n-1)}{\sqrt{n-1}}\right]
	=\bP\left[Z_{n-1}\le b+o(1)\right]\ton \Phi(b),
\end{align*}
where we defined $Z_n:=(2\Bin(n,1/2)-n)/\sqrt n$.
This proves the first claim.

Now consider the regime~\eqref{eq:regime_quermass_delta=1/2}. For $b<-c/\sqrt{2}$, we obtain that
\begin{align*}
k=\frac n2 +b\sqrt{\frac n4}+o(\sqrt{n})=d+\Big(\frac{c}{2}+\frac{b}{\sqrt{2}}\Big)\sqrt d +o(\sqrt d)<d
\end{align*}
for sufficiently large $d$. Thus, we can use the representation~\eqref{eq:Cover_Efron_quermass_binom_prob} of $\bE U_k(C_{n,d})$ and apply the central limit theorem for the binomial distribution to the denominator:
\begin{align*}
\bP[\Bin(n-1,1/2)\le d-1]
&	=\bP\left[Z_{n-1}\le \frac{2(d-1)-(n-1)}{\sqrt{n-1}}\right]\\
&	=\bP\left[Z_{n-1}\le -\frac{c}{\sqrt 2}+o(1)\right]\ton\Phi\Big(-\frac{c}{\sqrt 2}\Big).
\end{align*}
Similarly, the numerator of~\eqref{eq:Cover_Efron_quermass_binom_prob} converges to $\Phi(b)$ as $d\to\infty$, which yields
\begin{align*}
\lim_{d\to\infty}\bE 2U_k(C_{n,d})=1-\frac{\Phi(b)}{\Phi(-c/\sqrt 2)}, \quad b<- \frac{c}{\sqrt 2}.
\end{align*}
For $b>-c/\sqrt{2}$, we have that $k > d$ for sufficiently large $d$, which implies
\begin{align*}
\lim_{d\to\infty}\bE 2U_k(C_{n,d})=0.
\end{align*}
This completes the proof.
%In the case where $d/n\to\delta\in(0,1/2)$, as $d\to\infty$, and $k\in\NN$ is fixed, we can use~\eqref{eq:quermassint_CoverEfron_binomial_probs} and Theorem~\ref{theorem:v_k_CLT_Cover-Efron} (together with ) to obtain
%\begin{align*}
%\bE 2U_{d-k}(C_{n,d})
%%&	=\frac{\sum_{l=1}^k\bP[\Bin(n-1,1/2)=d-l]}{\bP[\Bin(n-1,1/2)\le d-1]}
%=\sum_{l=1}^k\bP[d-X_{n,d}=l]\ton\sum_{l=1}^k\bP[Z=l].
%\end{align*}
%
\end{proof}

We want to prove a similar result for $D_{n,d}$. To this end, we state the fromulas for the expected conic quermassintegrals of $D_{n,d}$ in the following lemma.

\begin{lemma}\label{lem:ConicQuerDonohoTanner}
Fix integers $0\le k\le d\leq n$. Let $D_{n,d}$ be a Donoho-Tanner random cone. Then, it holds that
\begin{align*}
\bE U_k(D_{n,d}) = \begin{cases}
{1\over 2^{n}}\sum_{l=k}^{d-1} \binom {n-1}l &: d-k\text{ even},\\
\frac{1}{2^{n}}\sum_{l=k}^{d-2}\binom{n-1}{l} +\frac{1}{2^{n}}\sum_{l=d}^n \binom nl 			&: d-k\text{ odd}.
\end{cases}
\end{align*}
\end{lemma}

\begin{proof}
Suppose at first that $d-k$ is even. Then, the conic Crofton formula~\eqref{eq:conic_Crofton} together with Lemma~\ref{lem:ConicIntVolDonohoTanner} yields
\begin{align*}
\bE U_k(D_{n,d})
&	=\bE\upsilon_{k+1}(D_{n,d})+\bE\upsilon_{k+3}(D_{n,d})+\ldots
	=\frac{1}{2^n}\left(\binom n{k+1}+\binom n{k+3}+\ldots+\binom{n}{d-1}\right).
\end{align*}
Using the relation $\binom{n}{k+1}=\binom {n-1}k+\binom {n-1}{k+1}$ yields the first claim. In the case where $d-k$ is odd, we similarly obtain
\begin{align*}
\bE U_k(D_{n,d})
&	=\frac{1}{2^n}\left(\binom n{k+1}+\binom n{k+3}+\ldots+\binom{n}{d-2}\right)+\Bigg(1-\sum_{j=0}^{d-1}\frac{1}{2^n}\binom nj\Bigg)\\
&	=\frac{1}{2^{n}}\sum_{l=k}^{d-2}\binom{n-1}{l} +\frac{1}{2^{n}}\sum_{l=d}^n\binom {n}l,
\end{align*}
which completes the proof.
\end{proof}

We can now present the analogue of Theorem \ref{thm:031220} for the Donoho-Tanner random cones. We remark that in this situation the case $\delta=1/2$ does not yield a proper limit for $\bE 2U_k(D_{n,d})$ and is therefore omitted. The case $\delta>1/2$ is illustrated in Figure~\ref{pic:U_k_Donoho}.

\begin{theorem}\label{thm:031220-2}
Let $D_{n,d}$ be a Donoho-Tanner random cone. Suppose that $n=n(d)$ and $k=k(d)$ are such that
\begin{align}\label{eq:regime_Donoho_distr_U_k}
\frac{d}{n}\to\delta\quad\text{and}\quad k=\frac n2+b\sqrt{\frac n4}+o(\sqrt{n}),\qquad\text{as }d\to\infty,
\end{align}
%\TG[$k$ in Ahängigkeit von $d$ geht nicht oder?? Geht nicht.]
for parameters $\delta\in(1/2,1)$ and $b\in\RR$. Then, it holds that
\begin{align*}
\lim_{d\to\infty}\bE 2U_k(D_{n,d})=1-\Phi(b).
\end{align*}
\end{theorem}

\begin{figure}[t]
\centering
\includegraphics[scale=0.6]{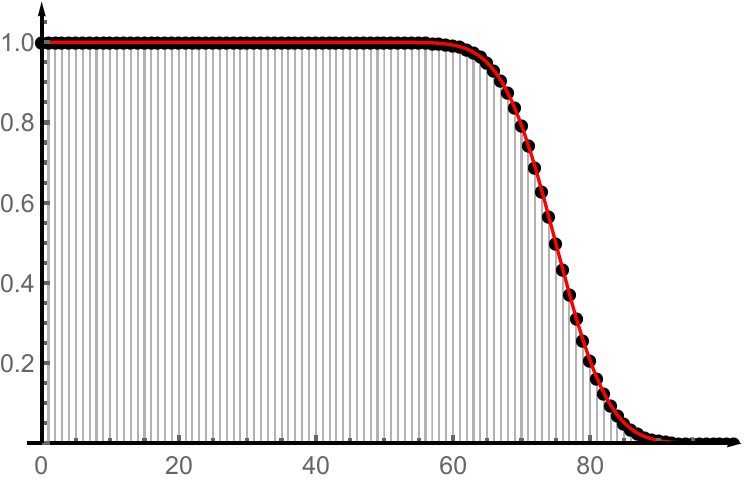}
\caption{Convergence in the regime~\eqref{eq:regime_Donoho_distr_U_k} with $d=100$ and $\delta=2/3$.
The \textbf{black points} represent the values of $\bE 2U_{k}(C_{n,d})$ for $k=0,1,\dots d$. The {\textbf{red curve}} is the tail function of the approximating distribution.
}\label{pic:U_k_Donoho}
\end{figure}

\begin{proof}
By Lemma~\ref{lem:ConicQuerDonohoTanner}, the expected conic quermassintegrals of $D_{n,d}$ can be written in terms of binomial probabilities:
\begin{align*}
\bE 2U_k(D_{n,d}) = \begin{cases}
\bP[k\le \Bin(n-1,1/2) \le d-1] &: d-k\text{ even},\\
\bP[k\le \Bin(n-1,1/2)\le d-2]+2\bP[\Bin(n,1/2)\ge d] &:d-k\text{ odd}.
\end{cases}
\end{align*}
Since $\delta>1/2$, the law of large number implies
\begin{align*}
\lim_{d\to\infty}\bP[\Bin(n,1/2)\ge d]=0\quad \text{and}\quad
\lim_{d\to\infty} \bP[\Bin(n-1,1/2) \ge  d-1] = 0.
\end{align*}
Using the central limit theorem for binomial random variables, we obtain
\begin{align*}
\bP[\Bin(n-1,1/2)\ge k]
	=\bP\left[Z_{n-1}\ge \frac{2k-(n-1)}{\sqrt {n-1}}\right]
	=\bP\left[Z_{n-1}\ge b+o(1)\right]\ton 1-\Phi(b),
\end{align*}
which yields the claim.
%where $Z_n:=(2\Bin(n,1/2)-n)/\sqrt n$. In the case where $d-k$ is odd, we use the law of large numbers and the fact that $\delta>1/2$ to observe that
%\begin{align*}
%\lim_{d\to\infty}\bP[\Bin(n,1/2)\ge d]=0, \qquad \text{and}\qquad
%\lim_{d\to\infty} \bP[\Bin(n-1,1/2) \ge  d-1] = 0.
%\end{align*}
%Thus, we also obtain
%\begin{align*}
%\lim_{d\to\infty}\bE 2U_k(D_{n,d})=1-\Phi(b)
%\end{align*}
%in the case where $d-k$ is odd. This completes the proof.
\end{proof}

To conclude this section, we take a look at a critical case in the convergence of $\bE U_{d-k}(C_{n,d})$. In Theorems~1.5 and 1.6 of~\cite{HugSchneiderThresholdPhenomena} Hug and Schneider proved for all fixed $k\in \NN$ that
\begin{align*}
\lim_{d\to\infty}\bE U_{d-k}(C_{n,d})=0
\end{align*}
if $n-2d$ is bounded from above (or a weaker, but more technical condition is satisfied), and that
\begin{align*}
\lim_{d\to\infty} \sqrt{d}\cdot\bE U_{d-k}(C_{n,d})=\frac{k}{\sqrt{\pi}}
\end{align*}
if $n=2d$. The next theorem considerably generalizes the last statement.

\begin{theorem}
Let $C_{n,d}$ be a Cover-Efron random cone. Suppose that $n=n(d)$ is such that
$$
n=2d+c\sqrt{d}+o(\sqrt{d}),\qquad\text{as }d\to\infty,
$$
for a parameter $c\in\RR$.
Then, it holds that
\begin{align*}
\lim_{d\to\infty}\sqrt{d}\cdot\bE 2U_{d-k}(C_{n,d})=\frac{e^{-c^2/4}}{\Phi(-c/\sqrt 2)}\cdot\frac k{\sqrt{\pi}}
\end{align*}
for all fixed $k\in\NN$.
\end{theorem}

\begin{proof}
Following~\eqref{eq:quermassint_cover-efron}, we can express the expected quermassintegral of $C_{n,d}$ in terms of binomial probabilities:
\begin{align}\label{eq:quermassint_CoverEfron_binomial_probs}
\bE 2U_{d-k}(C_{n,d})
	=1-\frac{\bP[\Bin(n-1,1/2)\le d-k-1]}{\bP[\Bin(n-1,1/2)\le d-1]}
	=\frac{\sum_{l=1}^k\bP[\Bin(n-1,1/2)=d-l]}{\bP[\Bin(n-1,1/2)\le d-1]}.
\end{align}
As we already observed, the central limit theorem for the binomial distribution implies that
\begin{align*}
\lim_{d\to\infty} \bP[\Bin(n-1,1/2)\le d-1] =  \Phi\Big(-\frac{c}{\sqrt 2}\Big).
\end{align*}
Furthermore, the local limit theorem~\eqref{eq:local_limit_thm_asym_neu} yields that
\begin{align*}
\bP[\Bin(n-1,1/2)=d-l]
&	= \frac{\sqrt{2}}{\sqrt{\pi(n-1})}\exp\left\{-\frac 12 \left(\frac{2(d-l)-(n-1)}{\sqrt{n-k-1}}\right)^2 \right\} + o\left(\frac{1}{\sqrt n}\right)\\
%&	= \frac{\sqrt{2}}{\sqrt{\pi(n-1})}\exp\left\{-\frac 12 \left(-\frac{c}{\sqrt 2}+o(1)\right)^2 \right\} +  o\left(\frac{1}{\sqrt n}\right)\\
&	= \frac{\sqrt{2}\cdot e^{-c^2/4}}{\sqrt{\pi n}}+ o\left(\frac{1}{\sqrt n}\right).
%& \hspace*{-2.2mm}\tosim  \frac{\sqrt{2}\cdot e^{-c^2/4}}{\sqrt{\pi n}}.
\end{align*}
Inserting both asymptotics into ~\eqref{eq:quermassint_CoverEfron_binomial_probs} yields
\begin{align*}
\bE 2U_{d-k}(C_{n,d})\tosim \frac{k\sqrt{2}\cdot e^{-c^2/4}}{\Phi(-c/\sqrt 2)\sqrt{\pi n}}
\end{align*}
%\begin{align*}
%\sum_{l=1}^k\bP[\Bin(n-1,1/2)=d-l]\tosim \frac{k\sqrt 2 \cdot e^{-c^2/4}}{\sqrt{\pi n}}.
%\end{align*}
Together with the observation that $\sqrt d/\sqrt n\to 1/\sqrt 2$, this yields the claim.
\end{proof}

\subsection{Large deviation principles}

Like in Section~\ref{section:LDP_f_k}, we provide a kind of large deviation principle for both $\bE\upsilon_k(D_{n,d})$ and $\bE\upsilon_k(C_{n,d})$ as $n$, $k$ and $d$ all tend to infinity in a linearly coordinated way. A similar theorem is stated for the expected quermassintegrals $\bE U_{d-k}(C_{n,d})$.

\begin{theorem}
Consider the Donoho-Tanner random cone $D_{n,d}$. Suppose that $n=n(d)$ and $k=k(d)$ are  such  that
\begin{align}\label{eq:regime_Donoho-Tanner_v_k}
{k\over d}\to\varrho\quad\text{and}\quad{d\over n}\to\delta,\qquad\text{as }d\to\infty,
\end{align}
for parameters $\varrho,\delta\in(0,1)$. Then, it holds that
\begin{align*}
\lim_{d\to\infty}{1\over d}\log\bE v_k(D_{n,d}) = -{1\over\delta}\log(2(1-\delta\varrho))-\varrho\log\varrho+\varrho\log\Big({1\over\delta}-\varrho\Big).
\end{align*}
\end{theorem}
\begin{proof}
Since $\varrho<1$ we can assume that $k\neq d$. Then, Lemma \ref{lem:ConicIntVolDonohoTanner} yields that
\begin{align*}
\lim_{d\to\infty}{1\over d}\log\bE v_k(D_{n,d}) = -\lim_{d\to\infty}{n\over d}\log 2 + \lim_{d\to\infty}{1\over d}\log{n\choose k}.
\end{align*}
Since $n/d\to {1\over\delta}$, relation \eqref{eq:LDPBinomialCoeff} shows that
\begin{align*}
\lim_{d\to\infty}{1\over d}\log\bE v_k(D_{n,d}) = -{1\over\delta}\log 2 + \cJ(\delta,\varrho).
\end{align*}
This yields the result.
\end{proof}

We turn now to a corresponding result for the Cover-Efron random cones.

\begin{theorem}
Let $C_{n,d}$ be a Cover-Efron random cone. Suppose that $k=k(d)$ and $n=n(d)$ are such that
\begin{align*}
{k\over d}\to\varrho\quad\text{and}\quad{d\over n}\to\delta,\qquad\text{as }d\to\infty,
\end{align*}
for parameters $\varrho,\delta\in(0,1)$. Then, it holds that
$$
\lim_{d\to\infty}{1\over d}\log\bE v_k(C_{n,d}) = \begin{cases}
{1\over\delta}\log{1\over\delta}-\varrho\log\varrho-\big({1\over\delta}-\varrho\big)\log\big({1\over\delta}-\varrho\big)- \frac{\log 2}{\delta} &: \delta>1/2,\\
\log\delta-\log(1-\delta)-\varrho\log\varrho+\varrho\log\big({1\over\delta}-\varrho\big)+{1\over\delta}\log\big({1-\delta\over 1-\delta\varrho}\big) &: \delta<1/2.
\end{cases}
$$
\end{theorem}
\begin{proof}
Since $\varrho<1$ by assumption it is sufficient to consider the case $k\neq d$. Using the explicit formula for $\bE v_k(C_{n,d})$ from~\eqref{eq:intr_vol_cover-efron} we can write
\begin{align*}
\bE v_k(C_{n,d}) = {{n\choose k}\over 2^n\,\bP[\Bin(n-1,1/2)\leq d-1]}
\end{align*}
and hence
\begin{align*}
\lim_{d\to\infty}{1\over d}\log{\bE v_k(C_{n,d})\over{n\choose k}} = -\log 2\lim_{d\to\infty}{n\over d} - \lim_{d\to\infty}{1\over d}\log\bP[\Bin(n-1,1/2)\leq d-1].
\end{align*}
Since $d/n\to\delta$, as $d\to\infty$, the first term is equal to $-{1\over\delta}\log 2$, while for the second term we have that
\begin{align*}
\lim_{d\to\infty}{1\over d}\log\bP[\Bin(n-1,1/2)\leq d-1] = \lim_{d\to\infty}{n-1\over d}{1\over n-1}\log\bP[\Bin(n-1,1/2)\leq n(\delta+o(1))].
\end{align*}
If $\delta<1/2$ we conclude from Cramer's theorem~\eqref{eq:Cramer<1/2} that
$$
\lim_{d\to\infty}{1\over d}\log\bP[\Bin(n-1,1/2)\leq d-1] = -{1\over\delta}\cI(\delta),
$$
while for $\delta>1/2$ the limit is zero. Combined with \eqref{eq:LDPBinomialCoeff} this yields that
\begin{align*}
\lim_{d\to\infty}{1\over d}\log\bE v_k(C_{n,d}) = \lim_{d\to\infty}{1\over d}\log{\bE v_k(C_{n,d})\over{n\choose k}} + \lim_{d\to\infty}{1\over d}\log{n\choose k} = {1\over\delta}(\cI(\delta)-\log 2) + \cJ(\delta,\varrho)
\end{align*}
if $\delta<1/2$ and
\begin{align*}
\lim_{d\to\infty}{1\over d}\log\bE v_k(C_{n,d}) = \cJ(\delta,\varrho)- \frac{\log 2}{\delta}
\end{align*}
for $\delta>1/2$. This completes the proof.
\end{proof}

In Theorem~1.4 of~\cite{HugSchneiderThresholdPhenomena} Hug and Schneider proved that
\begin{align*}
\lim_{d\to\infty}\bE U_{d-k}(C_{n,d})=0
\end{align*}
if $d/n\to\delta$ as $d\to\infty$ and a parameter $\delta\in(1/2,1)$. The next theorem gives the rate with which $\bE U_{d-k}(C_{n,d})$ converges to $0$.

\begin{theorem}
Let $C_{n,d}$ be a Cover-Efron random cone. Suppose that  $n=n(d)$ is such that
$$
\frac dn \longrightarrow\delta, \qquad\text{as }d\to\infty,
$$
for a parameter $\delta\in(1/2,1)$.
Then, for all fixed $k\in\NN$, it holds that
\begin{align*}
\lim_{d\to\infty}\frac{1}{d}\log \bE U_{d-k}(C_{n,d})
=
-\frac 1\delta(\log 2 + \delta \log \delta+(1-\delta)\log(1-\delta))
<0.
\end{align*}
\end{theorem}

\begin{proof}
Recall the representation
\begin{align*}
\bE 2U_{d-k}(C_{n,d})
%	=1-\frac{\bP[\Bin(n-1,1/2)\le d-k-1]}{\bP[\Bin(n-1,1/2)\le d-1]}
	=\frac{\sum_{l=1}^k\bP[\Bin(n-1,1/2)=d-l]}{\bP[\Bin(n-1,1/2)\le d-1]},
\end{align*}
of $\bE 2U_{d-k}(C_{n,d})$ in terms of binomial probabilities; see~\eqref{eq:quermassint_CoverEfron_binomial_probs}. The law of large numbers yields that
\begin{align*}
\lim_{d\to\infty}\bP[\Bin(n-1,1/2)\le d-1]=1.
\end{align*}
Using the asymptotic equivalence~\eqref{eq:asym_binomial_=}, we obtain that
\begin{align*}
\sum_{l=1}^k\bP[\Bin(n-1,1/2) =  d-l]
&	=\sum_{l=1}^k\bP\Big[\Bin(n-1,1/2)=\Big(\frac{d-l}{n-1}\Big)(n-1)\Big]\\
&	\hspace*{-2.2mm}\tosim \sum_{l=1}^k\frac{1}{\sqrt{2\pi(n-1)\delta(1-\delta)}}\exp\left\{-(n-1)\cI\Big(\frac{d-l}{n-1}\Big)\right\}\\
&	\hspace*{-2.2mm}\tosim \frac{k}{\sqrt{2\pi n \delta(1-\delta)}}\exp\left\{-(n-1)\cI(\delta+o(1))\right\}.
\end{align*}
Thus, we arrive at
\begin{align*}
\lim_{d\to\infty}\frac{1}{d}\log \bE U_{d-k}(C_{n,d})
&=	
-\lim_{d\to\infty}\frac nd\cI(\delta)
	=-\frac 1\delta (\log 2 + \delta \log \delta+(1-\delta)\log(1-\delta)),
\end{align*}
which completes the proof.
\end{proof}

\section{Limit theorems for the expected statistical dimension}\label{section:limit_stat_dimension}

The \textit{statistical dimension} $\Delta(C)$ of a cone $C\subset\RR^d$ measures its `true' or `intrinsic' size or complexity and is defined as
\begin{align*}
\Delta(C):=\sum_{j=0}^dj\upsilon_j(C).
\end{align*}
Note that if $L\subset\RR^d$ is an $\ell$-dimensional linear subspace for some $\ell\in\{0,1,\ldots,d\}$ then $\upsilon_k(L)=1$ if $k=\ell$ and $0$ otherwise, which yields $\Delta(L)=\ell\cdot\upsilon_\ell(L)=\ell$. The statistical dimension is in fact the canonical extension of the dimension of a subspace as it has been argued in Section 5.3 of \cite{ALMT14} that the statistical dimension is the only continuous, rotation-invariant and localizable valuation $\Delta(\,\cdot\,)$ on the space of cones in $\RR^d$ satisfying $\Delta(L)=\ell$ for $\ell$-dimensional subspaces. For more properties of the statistical dimension, we refer to~\cite{ALMT14}.

In this section, our goal is to understand the asymptotic behaviour of the expected statistical dimension of the random cones introduced in Section~\ref{section:random_cones}, as $d$ and $n$  tend to infinity in a coordinated way.

\subsection{The Donoho-Tanner random cones}

We start by considering the dual Donoho-Tanner random cone $D_{n,d}^\circ$. Using the well-known duality relation $\upsilon_k(C)=\upsilon_{d-k}(C^\circ)$ for a cone $C\subset\RR^d$, we can deduce the expected conic intrinsic volumes of $D_{n,d}^\circ$ from Lemma~\ref{lem:ConicIntVolDonohoTanner}:
\begin{align*}
\bE \upsilon_k(D_{n,d}^\circ)=
\begin{cases}
\frac{1}{2^n}\binom{n}{d-k}	&: k\in\{1,\dots,d\},\\
1-\sum_{j=1}^d\frac{1}{2^n}\binom{n}{d-j}	&:k=0.
\end{cases}
\end{align*}
Thus, the expected statistical dimension is given by
\begin{align}\label{eq:stat_sim_dual_Donoho}
\bE \Delta(D_{n,d}^\circ)=\sum_{j=0}^dj\cdot\bE\upsilon_j(D_{n,d}^\circ)=\sum_{j=0}^d\frac{j}{2^n}\binom{n}{d-j}={1\over 2^n}\sum_{l=0}^d(d-l)\binom{n}{l}.
\end{align}

The next theorem describes the asymptotic behaviour of $\bE\Delta(D_{n,d}^\circ)$ as $d$ and $n=n(d)$ tend to infinity simultaneously.

\begin{theorem}\label{theorem:stat_dim_dual_Donoho}
Let $D_{n,d}^\circ$ be a dual Donoho-Tanner random cone. Suppose that $n=n(d)$ is such that
$$
\frac dn \to\delta,\qquad\text{as }d\to\infty,
$$
for a parameter $\delta\in (0,1]$. Then, it holds that
\begin{align*}
\bE\Delta(D_{n,d}^\circ)\tosim d\Big(1-\frac{1}{2\delta}\Big),\quad\text{for }\delta\in(1/2,1],
\end{align*}
and
\begin{align*}
\lim_{d\to\infty}\bE\Delta(D_{n,d}^\circ)= 0,\quad\text{for }\delta\in(0,1/2).
\end{align*}
In the case where $\delta=1/2$, we have that
$$
\bE\Delta(D_{n,d}^\circ)=o(d),\qquad\text{as }d\to\infty.
$$
More precisely, if
$$
n= 2d+c\sqrt{d}+o(\sqrt{d}),\qquad\text{as }d\to\infty,
$$
for a parameter $c\in\RR$, then it holds that
\begin{align*}
\bE\Delta(D_{n,d}^\circ)\tosim \sqrt{d}\left(\frac{e^{-c^2/4}}{2\sqrt{\pi}}-\frac{c}{2}\Phi\Big(-\frac c{\sqrt 2}\Big)\right).
\end{align*}
\end{theorem}

\begin{proof}
Suppose that $d/n\to\delta$ as $d\to\infty$. We start with the case $\delta\in(1/2,1]$. Following~\eqref{eq:stat_sim_dual_Donoho}, we can rewrite the expected statistical dimension of $D_{n,d}^\circ$ in terms of binomial random variables:
\begin{align*}
\bE\Delta(D_{n,d}^\circ)
&	=\bE\big[(d-\Bin(n,1/2))\ind_{\{\Bin(n,1/2)\le d\}}\big]\\
&	=\bE\big[d-\Bin(n,1/2)]-\bE\big[(d-\Bin(n,1/2))\ind_{\{\Bin(n,1/2)> d\}}\big].
%&	\hspace*{-2.2mm}\overset{n\to\infty}{\sim} 2^n\cdot\bE[n-\Bin(n,1/2)],
\end{align*}
We have $d-n \leq d-\Bin(n,1/2)\leq d$ and it follows that $d-\Bin(n,1/2)\in [-Cd,Cd]$ for all $d\in \NN$ and a constant $C>0$. Thus, we have
\begin{align*}
\big|\bE\big[(d-\Bin(n,1/2))\ind_{\{\Bin(n,1/2)> d\}}\big]\big|
&	\le \bE\big[|d-\Bin(n,1/2)|\ind_{\{\Bin(n,1/2)> d\}}\big]\\
&	\le Cd\cdot\bP[\Bin(n,1/2)>d]\\
&	\le Cd\cdot\bP\left[\Bin(n,1/2)>n(\delta+o(1))\right],
\end{align*}
which converges to $0$ as $d\to\infty$ due to Cramer's theorem~\eqref{eq:Cramer>1/2} and the assumption $\delta>1/2$.
This yields
\begin{align*}
\lim_{d\to\infty}\frac{\bE\Delta(D_{n,d}^\circ)}{d}
=
\lim_{d\to\infty}\frac{\bE\big[d-\Bin(n,1/2)]}d=\lim_{d\to\infty}\Big(1-\frac n{2d}\Big) =1-\frac{1}{2\delta},
\end{align*}
which completes the proof of the case $\delta\in(1/2,1]$.

Now, we turn to the case $\delta\in(0,1/2)$. Consider the term
\begin{align*}
\sum_{l=0}^d(d-l)\binom{n}{l}
=\binom nd\sum_{k=0}^dk\frac{\binom n{d-k}}{\binom nd}
\end{align*}
and note that for every fixed $k\in \NN_0$,
\begin{align*}%\label{eq:asym_quotient_binom}
\frac{\binom{n}{d-k}}{\binom nd}=\frac{d(d-1)\cdot\ldots\cdot(d-k+1)}{(n-d+k)\cdot\ldots\cdot(n-d+1)}\tosim \frac{d^k}{(n-d)^k}\ton\Big(\frac\delta{1-\delta}\Big)^{k}.
\end{align*}
Now, we can apply the dominated convergence theorem, since there is an $\eps\in(0,1)$ such that for sufficiently large $d$, we have
\begin{align*}
k\frac{\binom{n}{d-k}}{\binom nd}\le k\Big(\frac{d}{n-d}\Big)^k\le k(1-\eps)^k,
\end{align*}
which is summable over $k$. This yields
\begin{align}\label{eq:asym_sum_delta<1/2}
\sum_{l=0}^d(d-l)\binom{n}{l}\tosim \binom{n}{d}\sum_{k=0}^\infty k\Big(\frac\delta{1-\delta}\Big)^{k}=\binom nd \frac{\delta(1-\delta)}{(1-2\delta)^2}
\end{align}
and together with~\eqref{eq:stat_sim_dual_Donoho}, we arrive at
\begin{align*}
\bE\Delta(D_{n,d}^\circ)\tosim\frac{\delta(1-\delta)}{(1-2\delta)^2}\frac{\binom nd}{2^n}=\frac{\delta(1-\delta)}{(1-2\delta)^2}\bP[\Bin(n,1/2)=d]\ton 0,
\end{align*}
which completes the proof of the case $\delta\in(0,1/2)$.

Now, let us turn to the case $\delta=1/2$, that is, $d/n(d)\to 1/2$, as $d\to\infty$. Our goal is to show that $\bE\Delta(D_{n,d}^\circ)/d\to 0$. To this end, fix some $\eps\in(0,1)$. For $n'(d):= [n(d)(1-\eps)]$, we have
\begin{align*}
\frac{d}{n'(d)}\ton \frac{1}{2(1-\eps)}>\frac 12.
\end{align*}
Thus, the first case of the theorem implies that $\bE \Delta(D_{n'(d),d}^\circ)\sim d\cdot\eps$, as $d\to\infty$.
For each sufficiently large $d\in\NN$, there is a natural coupling of $D_{n(d),d}^\circ$ and $D_{n'(d),d}^\circ$ such that $D_{n(d),d}^\circ\subset D_{n'(d),d}^\circ$ since $n'(d)<n(d)$.
Using the monotonicity of the statistical dimension, see~\cite[Proposition~3.1]{ALMT14}, yields
\begin{align*}
\limsup_{d\to\infty}\frac{\bE\Delta(D_{n(d),d}^\circ)}{d}\le \limsup_{d\to\infty}\frac{\bE\Delta(D_{n'(d),d}^\circ)}{d}=\eps.
\end{align*}
Letting $\eps$ approach zero from above yields the claim.

Now, let $n= 2d+c\sqrt{d}+o(\sqrt{d})$ for some $c\in\RR$. By defining $Z_n:=(2\Bin(n,1/2)-n)/\sqrt{n}$, we can write
\begin{align}\label{eq:asym_Delta_dual_Donoho_critical}
\bE\Delta(D_{n,d}^\circ)
&	=\bE\big[(d-\Bin(n,1/2))\ind_{\{\Bin(n,1/2)\le d\}}\big]\notag\\
&	=\bE\bigg[\Big(d-\sqrt{\frac n4}Z_n-\frac{n}{2}\Big)\ind_{\big\{Z_n\le \frac{2d-n}{\sqrt{n}}\big\}}\bigg]\notag\\
&	=\Big(d-\frac n2\Big)\bP\bigg[Z_n\le\frac{2d-n}{\sqrt{n}}\bigg]-\sqrt{\frac n4}\bE\Big[Z_n\ind_{\big\{Z_n\le \frac{2d-n}{\sqrt{n}}\big\}}\Big].
\end{align}
Following the central limit theorem for binomial random variables, we know that
\begin{align*}
\bP\bigg[Z_n\le\frac{2d-n}{\sqrt{n}}\bigg]=\bP\bigg[Z_n\le-\frac{c}{\sqrt 2}+o(1)\bigg]\ton \Phi\Big(-\frac{c}{\sqrt 2}\Big).
\end{align*}
Hence, it is left to determine the asymptotics of the term $\bE[Z_n\ind{\big\{Z_n\le \frac{2d-n}{\sqrt{n}}\big\}}]$. The central limit theorem for the binomial distribution yields the weak convergence of $Z_n$ to $N(0,1)$. Using Skorokhod's representation theorem we can assume, after passing to a different probability space, that $Z_n$ almost surely converges to a random variable $N(0,1)$, as $d\to\infty$. Thus, we also obtain
\begin{align*}
Z_n\ind_{\big\{Z_n\le\frac{2d-n}{\sqrt{n}}\big\}}\toas N(0,1)\ind_{\big\{N(0,1)\le -\frac{c}{\sqrt{2}}\big\}}.
\end{align*}
Since
\begin{align*}
\bE\bigg[\Big(Z_n\ind_{\big\{Z_n\le \frac{2d-n}{\sqrt{n}}\big\}}\Big)^2\bigg]\leq \bE\big[Z_n^2\big]= 1,
\end{align*}
the sequence $(Z_n\ind{\big\{Z_n\le \frac{2d-n}{\sqrt{n}}\}})_{d\ge 0}$ is uniformly integrable and we obtain the convergence of expectations
\begin{align*}
\bE\Big[Z_n\ind_{\big\{Z_n\le \frac{2d-n}{\sqrt{n}}\big\}}\Big]\ton\bE\Big[N(0,1)\ind_{\big\{N(0,1)\le -\frac{c}{\sqrt{2}}\big\}}\Big].
\end{align*}
The latter expectation can be further simplified as follows:
\begin{align*}
\bE\Big[N(0,1)\ind_{\big\{N(0,1)\le -\frac{c}{\sqrt{2}}\big\}}\Big]=\frac{1}{\sqrt{2\pi}}\int_{-\infty}^{-\frac{c}{\sqrt{2}}}xe^{-\frac{x^2}{2}}\textup{d}x=-\frac{1}{\sqrt{2\pi}}e^{-\frac{c^2}{4}}.
\end{align*}
Together with~\eqref{eq:asym_Delta_dual_Donoho_critical}, we finally arrive at
\begin{align*}
\bE\Delta(D_{n,d}^\circ)
&	\tosim\Big(d-\frac n2\Big)\Phi\Big(-\frac{c}{\sqrt 2}\Big)+\sqrt{\frac n4}\cdot\frac{1}{\sqrt{2\pi}}e^{-\frac{c^2}{4}}
	\tosim\sqrt{d}\left(\frac{e^{-c^2/4}}{2\sqrt{\pi}}-\frac{c}{2}\Phi\Big(-\frac c{\sqrt 2}\Big)\right).
%&	\tosim -\frac{c}{2}\sqrt{d}+\frac{e^{-c^2/4}}{\sqrt{2\pi}\Phi(-c/\sqrt{2})}\sqrt{\frac n4}
%&	\tosim \sqrt{d}\left(\frac{e^{-c^2/4}}{2\sqrt{\pi}\Phi(-c/\sqrt{2})}-\frac{c}{2}\right),
\end{align*}
where we used $n = 2d+c\sqrt{d}+o(\sqrt{d})$ in the last step. This completes the proof.
\end{proof}

By duality, the previous result implies  asymptotic results for the expected statistical dimension of the Donoho-Tanner random cone.

\begin{corollary}
Let $D_{n,d}$ be a Donoho-Tanner random cone. Suppose that $n=n(d)$ is such that
$$
\frac dn \to\delta,\qquad\text{as }d\to\infty,
$$
for a parameter $\delta\in (0,1]$. Then, it holds that
\begin{align*}
\bE\Delta(D_{n,d})\tosim
\begin{cases}
\frac{d}{2\delta}	&: \delta\in (1/2,1],\\
d					&: \delta\in (0,1/2]. \\
\end{cases}
\end{align*}
\end{corollary}

\begin{proof}
This follows directly from Theorem~\ref{theorem:stat_dim_dual_Donoho} using the duality relation~\cite[Eq.~(3.7)]{ALMT14}
\begin{align*}
\Delta(C)+\Delta(C^\circ)=d,
\end{align*}
for each $d$-dimensional cone $C\subset\RR^d$.
\end{proof}

\subsection{The Schl\"afli and Cover-Efron random cones}

Let $S_{n,d}$ be a Schl\"afli random cone.
By~\eqref{eq:intr_vol_Schlaefli} the expected statistical dimension of  $S_{n,d}$ is given by
\begin{align}\label{eq:stat_dim_Schlaefli}
\bE \Delta(S_{n,d})=\sum_{j=0}^dj\bE\upsilon_j(S_{n,d})=\sum_{j=0}^dj\frac{\binom{n}{d-j}}{C(n,d)}=\frac{\sum_{l=0}^{d}(d-l)\binom{n}{l}}{C(n,d)}.
\end{align}
The following theorem provides asymptotic results for $\bE\Delta(S_{n,d})$ as $d$ and $n$ tend to infinity simultaneously. This should be compared to Theorem~\ref{theorem:stat_dim_dual_Donoho}.

\begin{theorem}\label{theorem:exp_stat_dim_schlaefli}
Let $S_{n,d}$ be a Schl\"afli random cone. Suppose that $n=n(d)$ is such that
$$
\frac dn \to\delta,\qquad\text{as }d\to\infty,
$$
for a parameter $\delta\in (0,1]$. Then, it holds that
\begin{align*}
\bE\Delta(S_{n,d})\tosim
\begin{cases}
d(1-\frac{1}{2\delta})	&: \delta\in (1/2,1],\\
\frac{1}{2(1-2\delta)}			&: \delta\in (0,1/2).
\end{cases}
\end{align*}
In the case where $\delta=1/2$, we have that
$$
\bE\Delta(S_{n,d})=o(d),\qquad\text{as }d\to\infty.
$$
More precisely, if
$$
n= 2d+c\sqrt{d}+o(\sqrt{d}),\qquad\text{as }d\to\infty,
$$
for a parameter $c\in\RR$, then it holds that
\begin{align*}
\bE\Delta(S_{n,d})\tosim \sqrt{d}\left(\frac{e^{-c^2/4}}{2\sqrt{\pi}\Phi(-c/\sqrt{2})}-\frac{c}{2}\right).
\end{align*}
\end{theorem}

\begin{proof}
Following~\eqref{eq:stat_sim_dual_Donoho} and~\eqref{eq:stat_dim_Schlaefli}, we observe that
\begin{align}\label{eq:rel_stat_dim}
\bE\Delta(S_{n,d})=\bE\Delta(D_{n,d}^\circ)\cdot \frac{2^n}{C(n,d)}.
\end{align}
Consequently, some results for the expected statistical dimension of the Schl\"afli random cone follow from the corresponding results in Theorem~\ref{theorem:stat_dim_dual_Donoho}.

Now, let $d/n\to\delta$ for some $\delta\in(1/2,1]$. Then, the law of large numbers implies that
\begin{align*}
\lim_{d\to\infty} \bP[\Bin(n-1,1/2)\le d-1] = 1,
\end{align*}
Thus, we obtain
\begin{align*}
\frac{2^n}{C(n,d)}=\frac 1{\bP[\Bin(n-1,1/2)\le d-1]}\ton 1.
\end{align*}
Together with Theorem~\ref{theorem:stat_dim_dual_Donoho} this proves the case $\delta\in(1/2,1]$.
%
%Furthermore, we have
%\begin{align*}
%\sum_{l=0}^n(n-l)\binom{n}{l}
%&	=2^n\cdot\bE\big[(n-\Bin(n,1/2))\ind_{\{\Bin(n,1/2)\le n\}}\big]\\
%&	=2^n\cdot \left(\bE\big[(n-\Bin(n,1/2))]-\bE\big[(n-\Bin(n,1/2))\ind_{\{\Bin(n,1/2)> n\}}\big]\right).
%&	\hspace*{-2.2mm}\overset{n\to\infty}{\sim} 2^n\cdot\bE[n-\Bin(N,1/2)],
%\end{align*}
%Since $n-\Bin(N,1/2)$ is a.s.~asymptotically equivalent to $n(1-\frac{1}{2\delta})$ for $n\to\infty$, we know that $n-\Bin(N,1/2))\in [-cn,cn]$ for sufficiently large $n$ and a constant $c>0$. Thus, we have
%\begin{align*}
%\big|\bE\big[(n-\Bin(N,1/2))\ind_{\{\Bin(N,1/2)> n\}}\big]\big|\le \bE\big[|n-\Bin(N,1/2)|\ind_{\{\Bin(N,1/2)> n\}}\big]\le cn\bP[\Bin(N,1/2)>n],
%\end{align*}
%which converges to $0$ for $n\to\infty$ due to Cramer's theorem and $\delta>1/2$, see~\eqref{eq:Cramer>1/2}.
%This yields
%\begin{align}\label{eq:asym_sum_a<2}
%\sum_{l=0}^n(n-l)\binom{N}{l}\tosim 2^N \bE\big[(n-\Bin(N,1/2))]= 2^Nn\Big(1-\frac{1}{2\delta}\Big),
%\end{align}
%again following the law of large numbers.
%Together with~\eqref{eq:stat_dim_Schlaefli} and $C(N,n)\sim 2^N$ as $n\to\infty$, we arrive at
%\begin{align*}
%\bE \Delta(S_{N,n})\toas n\Big(1-\frac{1}{2\delta}\Big),
%\end{align*}
%completing the proof of the case $\delta\in(1/2,1]$.

Now, suppose that $\delta\in(0,1/2)$. At first, we want to determine the asymptotic behaviour of the term $C(n,d)$.
In the same way as in the proof of Theorem~\ref{theorem:stat_dim_dual_Donoho}, we obtain that
\begin{align*}
C(n+1,d+1)=2\sum_{l=0}^d\binom nl=2\binom nd\sum_{k=0}^d\frac{\binom n{d-k}}{\binom nd}\tosim 2\binom nd\sum_{k=0}^\infty \Big(\frac{\delta}{1-\delta}\Big)^k=2\binom nd \frac{1-\delta}{1-2\delta}.
\end{align*}
Together with~\eqref{eq:stat_dim_Schlaefli} and~\eqref{eq:asym_sum_delta<1/2}, we arrive at
\begin{align*}
\bE\Delta(S_{n,d})
=\frac{\sum_{l=0}^d{(d-l)\binom nl}}{2\sum_{l=0}^{d-1}\binom{n-1}l}
\tosim \frac{\binom nd}{\binom{n-1}{d-1}}\cdot\frac{\delta}{2(1-2\delta)}
= \frac nd \cdot\frac{\delta}{2(1-2\delta)}\tosim \frac{1}{2(1-2\delta)},
\end{align*}
which completes the proof of the case $\delta\in(0,1/2)$.

Now, we turn to the case $\delta=1/2$. Unfortunately, a monotonicity argument like the one used in a similar situation in the proof of  Theorem~\ref{theorem:stat_dim_dual_Donoho} fails for Schl\"afli cones. To work around this difficulty, fix an $\eps\in(0,1)$. We can split the sum on the right-hand side of~\eqref{eq:stat_dim_Schlaefli} to obtain
\begin{align*}
\frac{\bE\Delta(S_{n,d})}{d}=\frac{1}{d\cdot C(n,d)}\left(\sum_{l\in((1-\eps)d,d]}(d-l)\binom nl +\sum_{l\in[0,(1-\eps)d]}(d-l)\binom nl\right).
\end{align*}
For the first sum, we can use that $d-l\le \eps d$ and obtain
\begin{align*}
\limsup_{d\to\infty}\frac{1}{d\cdot C(n,d)}\sum_{l\in((1-\eps)d,d]}(d-l)\binom nl
&	\le \eps\cdot\limsup_{d\to\infty}\frac{\sum_{l=0}^d\binom nl}{C(n,d)}\\
&	\le \eps\cdot\limsup_{d\to\infty}\frac{\bP[\Bin(n,1/2)\le d]}{\bP[\Bin(n-1,1/2)\le d-1]}= \eps,
\end{align*}
where in the last step we used~\eqref{eq:k-faces_covefron_as_binomial_probs} and Theorem~\ref{theo:4.2}.
For the second sum, we use that $d-l\le d$ and obtain
\begin{align*}
\frac{1}{d\cdot C(n,d)}\sum_{l\in[0,(1-\eps)d]}(d-l)\binom nl
&	\le \frac{\sum_{l\in[0,(1-\eps)d]} \binom nl}{C(n,d)}=\frac{\bP[\Bin(n,1/2)\le (1-\eps)d]}{\bP[\Bin(n-1,1/2)\le d-1]}\ton 0.
\end{align*}
Note that we used the fact that the numerator converges to $0$ with an exponential rate, whereas it is not clear whether the denominator converges at all. This case was already treated in the proof of Theorem~\ref{theorem:limit_theorem_faces_cover-efron} (see the case $\delta=1/2$) and can be proven in the same way here. Letting $\eps$ approach zero from above yields that $\bE\Delta(S_{n,d})/d$ converges to $0$ as $d\to\infty$.

Now, we consider the critical case where $n=2d+c\sqrt{d}+o(\sqrt{d})$ as $d\to\infty$ for some constant $c\in\RR$.
By defining $Z_n:=(2\Bin(n,1/2)-n)/\sqrt{n}$, we can use the central limit theorem for binomial random variables to obtain
\begin{align*}
C(n,d)
	=2^n\bP\bigg[Z_{n-1}\le\frac{2(d-1)-(n-1)}{\sqrt{n-1}}\bigg]
	=2^n\bP\left[Z_{n-1}\le -\frac{c}{\sqrt 2}+o(1)\right]\tosim 2^n\Phi\Big(-\frac{c}{\sqrt 2}\Big).
\end{align*}
Hence, we also have
\begin{align*}\label{eq:asym_C(N,n)_2n}
\frac{2^n}{C(n,d)}\ton\frac{1}{\Phi(-c/\sqrt 2)}.
\end{align*}
Combining this with~\eqref{eq:rel_stat_dim} and Theorem~\ref{theorem:stat_dim_dual_Donoho} proves the claim.
\end{proof}

By the same duality argument as in the Donoho-Tanner case, Theorem~\ref{theorem:exp_stat_dim_schlaefli} yields the following corollary for the Cover-Efron random cone.

\begin{corollary}
Let $C_{n,d}$ be a Cover-Efron random cone. Suppose that $n=n(d)$ is such that
$$
\frac dn \to\delta,\qquad\text{as }d\to\infty,
$$
for a parameter $\delta\in (0,1]$. Then, it holds that
\begin{align*}
\bE\Delta(C_{n,d})\tosim
\begin{cases}
\frac{d}{2\delta}	&: \delta\in (1/2,1],\\
d					&: \delta\in (0,1/2]. \\
\end{cases}
\end{align*}
\end{corollary}
%\newpage %% AUTHOR: please comment out this line.  It serves only
%%   to demonstrate both types of header line in daj-template.pdf

%\section{Expansion estimates}

% More of the body of your paper goes here~\cite{bergelson-johnson-moreira}.

%%% AUTHOR: optional appendix here
%\appendix %% you may comment this out if no Appendix
%\section*{Appendix}
%\section{Improving the constants}
%Material is placed here as needed.

%%% AUTHOR: optional acknowledgments here
\section*{Acknowledgments} %%  you may comment this out if no Ackno
The authors are grateful to the anonymous reviewers for the careful reading of the manuscript and for useful suggestions.

%%% AUTHOR:
%%% Bibliography goes here. Note that the arXiv cannot process bibtex
%%% or biber bibliographies.  Example of acceptable bibliograpy format:
\bibliographystyle{amsplain}

%% AUTHOR: You can generate such a bibliography from a .bib file by
%% running pdflatex/bibtex/pdflatex/pdflatex and then pasting the .bbl file
%% between \begin{thebibliography} and \end{bibliography}

%%% AUTHOR: Include a short description of each author following the
%%% structure below. Use the same short tags used previously.
%%% Use \imageat{} and \imagedot{} instead of "@" and "." in
%%% email addresses-this replaces the symbols with graphics to avoid
%%% e-mail address harvesting from the .pdf file
\begin{dajauthors}
\begin{authorinfo}[TG]
  Thomas Godland\\
  Westf\"alische Wilhelms Universit\"at\\
  M\"unster, Germany\\
  thomas.godland\imageat{}wwu\imagedot{}de \\
  \url{https://www.uni-muenster.de/Stochastik/Arbeitsgruppen/Kabluchko/tgodland.shtml}
\end{authorinfo}
\begin{authorinfo}[ZK]
  Zakhar Kabluchko\\
 % Professor\\
  Westf\"alische Wilhelms Universit\"at\\
  M\"unster, Germany\\
  zakhar.kabluchko\imageat{}wwu\imagedot{}de \\
  \url{https://www.uni-muenster.de/Stochastik/Arbeitsgruppen/Kabluchko}
\end{authorinfo}
\begin{authorinfo}[CT]
  Christoph Th\"ale\\
 % Professor\\
  Ruhr-Universit\"at\\
  Bochum, Germany\\
  christoph.thaele\imageat{}rub\imagedot{}de\\
  \url{https://www.ruhr-uni-bochum.de/ffm/Lehrstuehle/Thaele/christoph.html}
\end{authorinfo}
%\begin{authorinfo}[andy]
%  Andrew Chi-Chih Yao\\
%  Professor\\
%  etc.
%\end{authorinfo}
\end{dajauthors}

\end{document}